\documentclass[11pt,reqno]{amsart}
\usepackage{graphicx}
\usepackage{color}
\usepackage{amsmath,amssymb,amsthm,amsfonts}
\usepackage{mathrsfs}
\usepackage{graphicx}
\usepackage{color}
\usepackage{multicol}
\def\R{\mathbb{R}}

\makeatletter{\normalsize }

\newcommand{\Rmnum}[1]{\expandafter\@slowromancap\romannumeral #1@}
\makeatother

\newtheorem{thm}{Theorem}[section]

{\huge }

\newtheorem{lemma}[thm]{Lemma}
\newtheorem{remark}{Remark}[section]
\newtheorem{theorem}[thm]{Theorem}

\usepackage[numbers,sort&compress]{natbib}

\begin{document}
\author{Hai-Yang Jin}
\address{School of Mathematics, South China University of Technology, Guangzhou 510640, P.R. China}
\email{mahyjin@scut.edu.cn}
\author{Jingyi Mai}
\address{School of Mathematics, South China University of Technology, Guangzhou 510640, P.R. China}
\email{majymai@163.com}

\title[Reaction-diffusion model  with dual-dependent motility]{Global boundedness and stabilization for a  three-component reaction-diffusion model with dual-dependent motility}
\begin{abstract}
In this paper, we consider the initial-boundary value problem of a three-component reaction-diffusion system with dual-dependent motility (which depends  on both \( v \) and \( m \)):  
\begin{equation}\tag{*}\label{1-1}
        \begin{cases}
             u_{t} = \Delta [\gamma ( v, m ) u ] + \alpha u m, & x \in \Omega, ~ t > 0,\\
             v_{t} =  \Delta v - v + (\theta_{1} + \theta_{2} m) u, & x \in \Omega, ~ t > 0,\\
             m_{t} = \Delta m - u m, & x \in \Omega, ~ t > 0,\\
             \nabla (u \gamma( v, m )) \cdot \nu = \nabla v \cdot \nu = \nabla m \cdot \nu = 0, & x \in \partial \Omega, ~ t > 0,\\
             (u, v, m)(x, 0) = (u_{0}, v_{0}, m_{0})(x), & x \in \Omega,
         \end{cases}
    \end{equation}
where \( \Omega \subset \mathbb{R}^2 \) is a bounded domain with a smooth boundary. Here \( \alpha \), \( \theta_1 \) and \( \theta_2 \) are non-negative constants, and \( \nu \) denotes the outward unit normal vector on \( \partial\Omega \). The random motility function \( \gamma(v,m) \), depending on both the chemical concentration \( v \) and the nutrient level \( m \), satisfies:  
\begin{itemize}
\item \( \gamma (v, m) \in C^{3} \left([0,\infty) \times [0,\infty)\right) \), \( \gamma (v, m) > 0 \), \( \gamma_{v} (v, m) \leq 0 \), \( \gamma_{m} (v, m) \geq 0 \) on \( [0,\infty) \times [0,\infty) \); moreover, \( \lim\limits_{v \to + \infty} \gamma (v, m) = 0 \) for any \( m \geq 0 \).  
\end{itemize}
For \( \theta_1 = 0 \), we first establish the existence of global classical solutions with uniform-in-time bounds for the system \eqref{1-1}, and further prove that the solution \( (u,v,m) \) converges to the constant steady state \( (u_*,0,0) \) as \( t \to \infty \), where \( u_{*} = \frac{1}{| \Omega |} \left(\Vert u_{0} \Vert_{L^{1}} + \alpha \Vert m_{0} \Vert_{L^{1}}\right) \). For \( \theta_1 > 0 \), we demonstrate the global boundedness of solutions under appropriate decay conditions on \( \gamma(v,m) \). We further show that there exists \( \theta_* > 0 \) such that if \( 0 < \theta_1 \leq \theta_* \), all solutions converge to \( (u_*,\theta_1 u_*,0) \) as \( t \to \infty \) with \( u_* \) defined as above.

Our analysis employs refined energy estimates, moser iteration techniques, and careful exploitation of the structural properties of \( \gamma(v,m) \). These results characterize the solution behavior in both cases (\( \theta_1 = 0 \) and \( \theta_1 > 0 \)), highlighting the crucial role of dual-dependent motility in the system's long-term dynamics.

\end{abstract}



\maketitle

\numberwithin{equation}{section}
\section{Introduction}
In this paper, we consider the  initial-boundary value problem for the following three-component reaction-diffusion system  with density-dependent motility
    \begin{equation}\label{sys-1}
        \begin{cases}
             u_{t} = \Delta [\gamma ( v, m ) u ] + \alpha u m, & x \in \Omega, ~ t > 0,\\
             v_{t} =  \Delta v - v + (\theta_{1} + \theta_{2} m) u, & x \in \Omega, ~ t > 0,\\
             m_{t} = \Delta m - u m, & x \in \Omega, ~ t > 0,\\
             \nabla (u \gamma( v, m )) \cdot \nu = \nabla v \cdot \nu = \nabla m \cdot \nu = 0, & x \in \partial \Omega, ~ t > 0,\\
             (u, v, m)(x, 0) = (u_{0}, v_{0}, m_{0})(x), & x \in \Omega,
         \end{cases}
    \end{equation}
where $\Omega \subset \mathbb{R}^{n} (n\geq 1)$ is a  bounded domain with smooth boundary and $\nu$ stands for the outward unit normal vector on $\partial \Omega$. The parameters $\alpha\geq 0$ and $\theta_{1}, \theta_{2} \geq 0$. The motility response function $\gamma(v,m)$ depends on both chemical concentration $v$ and the nutrient $n$, exhibiting chemical-immobilizing and resource-mobilizing properties in the sense that $$\gamma_v(v,m)\leq 0\ \ \mathrm{and}\ \ \gamma_m(v,m)\geq 0.$$ 
Such kind of dual-dependent motility function $\gamma(v,m)$ was first derived in \cite{Kim-P} by considering  dweller-wanderer interactions, where the organisms adjust their movement patterns (e.g., velocity, turning frequency, residence time) in response to dual environmental gradients: resource availability and pheromone concentration. This regulation hinges on two opposing yet complementary responses, which collectively optimize the trade-off between local exploitation (by dwellers) and global exploration (by wanderers) in a population. For the system \eqref{sys-1}, existing results \cite{Winkler-M3AS-2025} primarily focus on simplified cases where the motility function $\gamma$ depends exclusively on a single variable: either $\gamma(v, m) = \phi(v)$ or $\gamma(v, m) = \varphi(m)$. 
  
\textbf{Case 1: $\gamma (v, m)=\phi(v)$.} In this case, when $\theta_2=0$, the system \eqref{sys-1} reduces to the following form:
 \begin{equation}\label{OR}
        \begin{cases}
             u_{t} = \Delta [\phi(v) u] + \alpha u m, & x \in \Omega, ~ t > 0,\\
             v_{t} =  \Delta v - v + \theta_{1}u, & x \in \Omega, ~ t > 0,\\
             m_{t} = \Delta m - u m, & x \in \Omega, ~ t > 0,\\
             \nabla (\phi(v) u) \cdot \nu = \nabla v \cdot \nu = \nabla m \cdot \nu = 0, & x \in \partial \Omega, ~ t > 0,\\
             (u, v, m)(x, 0) = (u_{0}, v_{0}, m_{0})(x), & x \in \Omega,
         \end{cases}
    \end{equation}
which was first proposed in \cite{Liu-Science-2011} to gain a quantitative understanding of the periodic strip pattern formation. For the system \eqref{OR}, the global existence, uniform boundedness, and stabilization of classical solutions have been intensively investigated recently by virtue of two distinct methods: energy estimate \cite{TW-M3AS-2017,JSW-2020-JDE,JW-PAMS-2020} and comparison method \cite{FJ-2020-JDE,FJ-2021-AAM,FJ-2021-CVPDE,FS-2022-NA,JL-2021-JDE,JLZ-2022-CPDE}. More precisely, when $n=2$ and $\alpha=0$, the first result on the global boundedness of classical solution for the system \eqref{OR} was obtained by Tao and Winkler \cite{TW-M3AS-2017} under the assumption that $0<\phi_0\leq\phi(v)\leq \phi_1$ by using dual arguments. Based on the ingenious  comparison method, Fujie and Jiang \cite{FJ-2021-CVPDE} relax the assumption $0<\phi_0\leq\phi(v)\leq \phi_1$ and establish global existence for any positive, non-increasing motility function $\phi(v)$ that converges to zero at infinity, even allowing the function to vanish asymptotically.  Moreover, for any non-increasing motility function decaying slower than any negative exponential function at infinity, the uniform-in-time bound of classical solution was obtained \cite{AY-2019-Nonlinearity,FJ-2021-AAM}. If $\phi(v)=e^{-\chi v}$ for $\chi>0$, a critical mass phenomenon was found \cite{BLT-2021-JLMS,FJ-2020-JDE,FJ-2021-CVPDE,JW-PAMS-2020} in the sense that  if $\int_\Omega u_0 dx <\frac{4\pi}{\chi\theta_1}$, the solution with uniform-in-time bound exists, while the solution will blow up at infinity time if $\int_\Omega u_0 dx >\frac{4\pi}{\chi\theta_1}$.
 

For higher dimensions ($n\geq 3$), it has been proved that the system \eqref{OR} with $\alpha=0$ has a unique global classical solution for arbitrary positive motility function \cite{FS-2022-NA}. The uniform-in-time boundedness of classical solution was also studied extensively for both the simplify parabolic-elliptic case \cite{AY-2019-Nonlinearity,FJ-2021-AAM,JL-2021-JDE}  or the full parabolic-parabolic case \cite{FJ-2021-AAM,FJ-2021-CVPDE}. This results imply that $\phi(s)\sim  s^{-k}$ as $s\to\infty$ with any  $k<\frac{n}{n-2}$ seems to the optimal condition for the existence of  global classical  of solution with uniform-in time bound when $n\geq 3$.  The results on global existence, uniformly boundedness and  stabilization of constant steady state also been extended to the case $\alpha>0$, see \cite{JSW-2020-JDE,FJ-2021-CVPDE}. At last, we should mention that except the global boundedness and stabilization of classical solution for the system \eqref{OR}, there also some results on the weak solution \cite{DKTY-2019-NA,TW-M3AS-2017}. There also some interesting results on the system \eqref{OR} with logistic growth including boundedness \cite{FJ-2020-JDE,JKW-2018-SIAM}, non-constant steady state \cite{WX-2021-IMA,MPW-PD-2020} and traveling wave \cite{LW-JDE-2021}.

\textbf{Case 2: $\gamma (v, m)=\varphi(m)$.} In this case, the component $v$ decouples from $u, m$ and hence the system \eqref{sys-1} becomes 
\begin{equation}\label{OR-2}
        \begin{cases}
             u_{t} = \Delta [\varphi (m) u] + \alpha u m, & x \in \Omega, ~ t > 0,\\
             m_{t} = \Delta m - u m, & x \in \Omega, ~ t > 0,\\
             \nabla (\varphi(m)u) \cdot \nu = \nabla m \cdot \nu = 0, & x \in \partial \Omega, ~ t > 0,\\
             (u, m)(x, 0) = (u_{0},  m_{0})(x), & x \in \Omega.
         \end{cases}
    \end{equation}
The model \eqref{OR-2} has been used to describe the bacterial populations how adjust their motility in response to the concentration of surrounding nutrients consumed upon contact. For instance, motility may increase under starvation conditions \cite{EK-BMB-2013} or decrease when nutrients are scarce, such as in poor agar environments \cite{LMP-PA-2013}. Compared with chemotaxis-production interaction system \eqref{OR}, few results are  available for the system \eqref{OR-2} due to the lack of some favorable structural properties.  The solution behaviors of \eqref{OR-2} depend on the fact that   $\varphi(m)$ is degenerate or not. More precisely, if $\varphi'(m)\leq 0$, the system \eqref{OR-2} can be viewed as the simplest predator–prey system with density-dependent motility \cite{PO-AN-1987,JW-EJAM-2021}. If $1\leq n\leq 2$, it has been proved in \cite{LY-NARWA-2023,WW-APL-2022} that the system \eqref{OR-2} has a unique  global classical  solution $(u,m)$, which converges to the constant steady state $(u_*,0)$ with $u_*=\frac{1}{|\Omega|}\int_\Omega u_0 dx+\frac{\alpha}{|\Omega|}\int_\Omega m_0dx$. The results obtained in \cite{LY-NARWA-2023,WW-APL-2022,LW-CMS-2023} are based on the fact $\varphi(m)$  is strictly positive throughout $[0,\infty)$ and hence non-degenerate occurs. In contrast to the case $\varphi'(m)\leq 0$, the solution behavior is much more complex due to the possible degeneracy of $\varphi(m)$. In fact, if $\varphi(m)>0$ but $\varphi(0)=0$, the global classical solution of the system \eqref{OR-2} still exist in the case of $1\leq n\leq 2$, which however will approach a non-constant steady state $(u_\infty,0)$ as $t\to\infty $ when $m_0$  is appropriately small \cite{W-AIHP-2024}. These results in \cite{LY-NARWA-2023,WW-APL-2022,LW-CMS-2023} imply that the degeneracy of signal-dependent motility plays an important role for the existence of non-constant large-time patterns in related taxis-consumption systems. Beyond the existence of global classical solution for the system \eqref{OR-2} in lower dimensions ($1\leq n\leq 2$), the existence of global very weak solution for any dimensions has also been established in \cite{LW-CMS-2023,W-Nonlinearity-2023} for any $\phi\in C^3([0,\infty))$ satisfying $\varphi>0$ on $[0,\infty)$.

To our  knowledge, there is few result on the chemotaxis model with motility response function $\gamma(v,m)$ depending on both chemical concentration $v$ and the nutrient $m$. In this paper, we assume that the motility function $\gamma(v, m)$ satisfies the following properties:
\begin{itemize}
    \item[(H1)] $\gamma (v, m) \in C^{3} ([0,\infty) \times [0,\infty))$, $\gamma (v, m) > 0$ and $\gamma_{v} (v, m) \leq 0$, $\gamma_{m} (v, m) \geq 0$ on $[0,\infty) \times [0,\infty)$. Moreover, for any $m\geq0$, it holds that 
   \begin{equation*}
        \lim\limits_{v \to + \infty} \gamma (v, m) = 0.
    \end{equation*}
\end{itemize}
 In this paper, we aim to address several fundamental questions regarding the behavior of solutions to the three-component reaction-diffusion system with a density-dependent motility function $\gamma(v,m)$ that depends on both the chemical concentration $v$ and the nutrient $m$. Specifically, we investigate the conditions on $\gamma(v,m)$ under which solutions exist globally in time, whether these solutions stabilize to constant or non-constant steady states as time evolves, and how the interplay between chemical sensing and nutrient sensing influences the formation of spatial patterns. These questions are motivated by the need to extend previous studies that focused on simplified cases where $\gamma$ depends solely on $v$ or $m$, and to provide a more comprehensive understanding of systems where the motility is regulated by multiple environmental factors. 

Throughout this paper, we assume that the initial value satisfies
\begin{equation}\label{ID}
    (u_{0}, v_{0}, m_{0}) \in [W^{1, \infty} (\Omega)]^{3}, \quad u_{0} \geq 0, \quad u_0 \not\equiv 0, \quad m_{0} \geq 0, \quad  v_{0} > 0 \quad \text{in} \ \bar{\Omega}.
\end{equation}
Next, we will investigate the global boundedness and large time behavior of classical solutions to \eqref{sys-1} for two distinct cases: $\theta_1=0$ and $\theta_1>0$. If $\theta_1=0$, then we have the following results on the global dynamics for the solution behavior.
\begin{theorem}[Boundedness and stabilization: $\theta_1=0$]\label{GBS-1}
    Let $\Omega \subset\R^2$ be a bounded domain with smooth boundary and $\theta_1=0$. Suppose that $\gamma$ satisfies the assumption (H1), and that the initial data $(u_{0}, v_{0}, m_{0})$ satisfies \eqref{ID}. Then the system \eqref{sys-1} has a unique  non-negative global classical solution $(u, v, m) \in [C(\bar{\Omega} \times [0,\infty)) \cap C^{2,1}(\bar{\Omega} \times (0,\infty))]^{3}$ satisfying
    \begin{equation*}
        \|u(\cdot,t)\|_{L^\infty} + \|v(\cdot,t)\|_{W^{1,\infty}} + \|m(\cdot,t)\|_{W^{1,\infty}} \leq C,
    \end{equation*}
where $C>0$ is a constant independent of $t$. Moreover, it holds that  
    \begin{equation*}
        \lim_{t \to \infty} (\|u(\cdot, t) - u_{*}\|_{L^{\infty}} + \|v(\cdot, t)\|_{L^{\infty}} + \|m(\cdot, t)\|_{L^{\infty}}) = 0,
    \end{equation*}
    where $u_{*} = \frac{1}{| \Omega |} (\|u_{0}\|_{L^{1}} + \alpha \|m_{0}\|_{L^{1}})$.
\end{theorem}
\begin{remark}
The results in Theorem \ref{GBS-1} imply that the solution will not blow up in two-dimensional spaces and no pattern formation occurs if the chemical signal produced by food.
\end{remark}
\begin{remark}
In Theorem \ref{GBS-1}, we assume that $\gamma(v,m)\in C^3([0,\infty)\times [0,\infty))$. However, when $\theta_1=0$, from Theorem \ref{GBS-1}, we know that $v$ converges to zero as $t\to \infty$. This raises an interesting problem: what can be said about the global dynamics for the case $\gamma(v,m)=m/v^k$?
\end{remark}
\begin{remark}
In Theorem \ref{GBS-1}, we only establish the global existence and boundedness of solutions in two-dimensional spaces. An intriguing open question is whether we can derive the global boundedness of solutions for the system \eqref{sys-1} with $\theta_1=0$ in higher dimensions.
\end{remark}
If $\theta_1>0$, the behavior of solutions exhibits distinct characteristics. We now elaborate on our findings regarding the global boundedness and stabilization of solutions for the system \eqref{sys-1} with $\theta_1>0$.
\begin{theorem}[Boundedness and stabilization: $\theta_1>0$]\label{GBS-2}
    Let $\Omega \subset\R^2$ be a bounded domain with smooth boundary and $\theta_1>0$. If $\gamma$ satisfies the assumption (H1) and there exists  $\chi>0$ such that  
    \begin{equation}\label{GBS-2*}
        \liminf_{s \to + \infty} e^{\chi s} \gamma (s, m) > 0, \quad \text{for all} \ m \geq 0,
    \end{equation}
    and the initial data $(u_{0}, v_{0}, m_{0})$ satisfies \eqref{ID} with
    \begin{equation*}
        \|u_{0} + \alpha m_{0}\|_{L^{1}} < \frac{4 \pi}{\chi(\theta_1+\theta_2 \|m_0\|_{L^\infty})}.
    \end{equation*}
    Then the system \eqref{sys-1} has a unique global classical solution with uniform-in-time bound. Moreover, there exists a constant $\theta_*>0$ such that if $0<\theta_1\leq \theta_*$, then 
    \begin{equation*}
        \lim_{t \to \infty} \left( \Vert u(\cdot, t) - u_{*} \Vert_{L^{\infty}} + \Vert v(\cdot, t) - \theta_{1} u_{*} \Vert_{L^{\infty}} + \Vert m(\cdot, t) \Vert_{L^{\infty}} \right) = 0,
    \end{equation*}
    where $u_{*} = \frac{1}{| \Omega |} (\Vert u_{0} \Vert_{L^{1}} + \alpha \Vert m_{0} \Vert_{L^{1}})$.
            
    In particular, if  $\gamma$ satisfies the assumptions (H1) and \eqref{GBS-2*} for any $\chi>0$, then the system  \eqref{sys-1} still has a unique classical solution with uniform-in-time bound for any initial data $(u_{0}, v_{0}, m_{0})$ satisfying \eqref{ID}.
\end{theorem}

\begin{remark}
Comparing the results in Theorem \ref{GBS-1} ($\theta_1=0$) and Theorem \ref{GBS-2} ($\theta_1>0$), we observe that the presence of chemical production ($\theta_1>0$) leads to different asymptotic behavior, with the chemical concentration $v$ stabilizing to $\theta_1 u_*$ rather than zero.
\end{remark}

\section{Local existence and Preliminaries}
We first establish the existence of local solutions to \eqref{sys-1} by employing the Schauder fixed point theorem along with the parabolic regularity theory \cite{JKW-2018-SIAM} or the Amann's theorem \cite{Amann-1993-FS,Amann-1990-DIE}. For brevity, we omit the details of the proof.
\begin{lemma}[Local existence]\label{LS}
    Let $\Omega \subset \mathbb{R}^{2}$ be a bounded domain with smooth boundary. Assume that $\gamma$ satisfies (H1) and the initial data $(u_{0}, v_{0}, m_{0})$ satisfies \eqref{ID}. 
    Then there exists $T_{\max} \in (0, \infty]$ such that the problem \eqref{sys-1} has a unique non-negative classical solution $(u, v, m) \in [C(\bar{\Omega} \times [0,T_{\max})) \cap C^{2,1}(\bar{\Omega} \times (0,T_{\max}))]^{3}$. If $T_{\max} < \infty$, then
    \begin{equation*}
        \limsup_{t \nearrow T_{\max}} \Vert u(\cdot, t) \Vert_{L^{\infty}} = \infty.
    \end{equation*}
\end{lemma}
With the local existence of solution in hand, we establish some basic estimates of solutions for the system \eqref{sys-1}.
\begin{lemma}\label{L1}
    Let $(u, v, m)$ be the solution of \eqref{sys-1}  obtained in Lemma \ref{LS}, then we have
    \begin{equation}\label{L1-1}
        0 < \Vert u_{0} \Vert_{L^{1}} \leq \Vert u(\cdot,t) \Vert_{L^{1}} \leq \Vert u_{0} + \alpha m_{0} \Vert_{L^{1}}, \quad \text{for all}\ t \in [0,T_{\max}),
    \end{equation}
    and 
    \begin{equation}\label{L1-2}
        0 \leq m(x, t) \leq \|m_0\|_{L^{\infty}}:=m^*, \quad (x, t) \in \bar{\Omega} \times [0, T_{\max}),
    \end{equation}
    as well as
    \begin{equation}\label{L1-3}
        \Vert m(\cdot,t) \Vert_{L^{\infty}}\ \text{is decreasing in}\ t.
    \end{equation}
\end{lemma}

\begin{proof}
    Integrating the first equation of \eqref{sys-1} and using the non-negativity of $u$ and $m$, we derive 
    \begin{equation}\label{L1-4}
        \frac{d}{dt}\int_\Omega u =\alpha\int_\Omega um,
    \end{equation}
    which gives 
    \begin{equation}\label{L1-5}
        \|u(\cdot,t)\|_{L^1}\geq \|u_0\|_{L^1}>0.
    \end{equation}
    On the other hand, multiplying the third equation of \eqref{sys-1} by $\alpha$, and integrating it by parts and adding it to \eqref{L1-4}, one has
    \begin{equation*}
        \frac{d}{dt}\int_\Omega (u+\alpha m)=0
    \end{equation*}
    and hence 
    \begin{equation}\label{L1-6}
        \|u(\cdot,t)\|_{L^1}\leq \int_\Omega (u+\alpha m)=\int_\Omega (u_0+\alpha m_0).
    \end{equation}
    Then the combination of \eqref{L1-5} and \eqref{L1-6} gives \eqref{L1-1}.

    Applying the parabolic comparison principle and noting the non-negativity of $m$ and $u$, we obtain \eqref{L1-2} directly.

At last,  we shall employ the maximum principle to prove \eqref{L1-3}. For any fixed $t_0\in[0,T_{\max})$ and  $(x,t)\in \Omega \times [0,T_{\max}-t_0)$, we set $$z(x,t)=m(x,t+t_0)-\|m(\cdot,t_0)\|_{L^\infty}.$$ 
Then we can check that $z$ satisfies 
    \begin{equation*}
        \left\{
        \begin{aligned}
            &z_t - \Delta z + uz \leq 0, &&\quad x\in \Omega,\ t\in (0,T_{\max}-t_0),\\
            &\nabla z \cdot \nu =0, &&\quad x\in \partial \Omega, \ t \in (0,T_{\max}-t_0),\\
            &z(x,0)=m(x,t_0)-\|m(\cdot,t_0)\|_{L^\infty}\leq 0,&& \quad x\in\Omega.
        \end{aligned}
        \right.
    \end{equation*}
    Thus, invoking the maximum principle, we deduce that
    \begin{equation}
    m(x,t+t_0) \leq \|m(\cdot,t_0)\|_{L^\infty} \ \ \ \mathrm{for}\ \ (x,t) \in \Omega \times [0,T_{\max}-t_0),
   \end{equation}
   which gives \eqref{L1-3}, due to  the arbitrary of $t_0$.
\end{proof}

The following lemma has been  proven in \cite[Lemma 2.3]{F-2016} and \cite[Lemma 3.3]{Wang-2021-MMAS}.
\begin{lemma}\label{LB}
    Let $\Omega \subset \mathbb{R}^{2}$ be a bounded domain.
    Assume that $f \in C(\bar{\Omega})$ is a non-negative function satisfying $\|f\|_{L^{1}} = \Lambda > 0$. If $z\in C^{2}(\bar{\Omega}) $ is a solution to
    \begin{equation*}
        \begin{cases}
            - \Delta z + z = f, & \quad x \in \Omega, \\
		\frac{\partial z}{\partial \nu} = 0, & \quad x \in \partial \Omega,\\
        \end{cases}
    \end{equation*}
    then one has
    \begin{equation}\label{LB-2}
        z \geq  C_1 \|f\|_{L^1} > 0 \quad \text{in}\ \Omega,
    \end{equation}
    and 
    \begin{equation}\label{LB-3}
        \Vert z \Vert_{L^{q}} \leq C_2 \Vert f \Vert_{L^{1}}, \quad \text{for}\ 1 \leq q <\infty
    \end{equation}
as well as 
 \begin{equation}\label{LB-4}
        \int_{\Omega} e^{\sigma z} d x \leq C_3 e^{C_4 \sigma \Lambda}, \  \ \mathrm{if}\ \  \sigma <\frac{4\pi}{\Lambda},
    \end{equation}
    where $C_i(i=1,2,3,4)$ are positive constants.
\end{lemma}
 
\begin{proof}
    The proof of \eqref{LB-2} can be found in \cite[Lemma 2.3]{F-2016}. Next, using the representation of resolvents and the smoothing estimates for Neumann semigroup, we derive
    \begin{equation*}
        \Vert z \Vert_{L^{q}} \leq c_1 \Vert (I - \Delta)^{- 1} f\Vert_{L^{q}} \leq c_2 \int_{0}^{\infty} e^{- t} \Vert e^{t \Delta} f \Vert_{L^{q}} d t \leq c_3 \left(\int_{0}^{\infty} e^{- t} t^{- 1 + \frac{1}{q}} d t\right) \Vert f \Vert_{L^{1}},
    \end{equation*}
    and then \eqref{LB-3} follows by noting $\int_{0}^{\infty} e^{- t} t^{- 1 +\frac{1}{q}} d t\leq c_4$ due to $1 \leq q <+\infty$. The estimate \eqref{LB-4} is proved in \cite[Lemma 3.3]{Wang-2021-MMAS}. Hence, we complete the proof of Lemma \ref{LB}.
\end{proof}

At last, we state a basic lemma will be used to establish the boundedness of solutions.
\begin{lemma}[\cite{LW-2015-CPDE}]\label{LB3}
    Let $T > 0$, $\tau \in (0, T)$, $a > 0$ and $b > 0$. Suppose that $y : [0, T) \rightarrow [0, \infty)$ is absolutely continuous and fulfills
    \begin{equation*}
        y'(t) + a y(t) \leq h(t), \quad \text{for all} \ t \in (0, T),
    \end{equation*}
    with some non-negative functions $h \in L_{loc}^{1}([0, T))$ satisfying
    \begin{equation*}
        \int_{t}^{t + \tau} h(s) ds \leq b, \quad \text{for all} \ t \in [0, T - \tau).
    \end{equation*}
    Then
    \begin{equation*}
        y (t) \leq \max \left\{ y (0) + b, \frac{b}{a \tau} + 2 b \right\}, \quad \text{for all} \ t \in (0, T).
    \end{equation*}
\end{lemma}

\section{Proof of Theorem \ref{GBS-1}: Boundedness and stabilization with $\theta_1=0$}
In this section, we shall study the global boundedness and global stabilization for the system \eqref{sys-1} with $\theta_1=0$ under the assumption that $\gamma(v,m)$ satisfies (H1). From Lemma \ref{LS}, we obtain $v \geq 0$, which combined with the assumption (H1) of $\gamma$ and \eqref{L1-2}, gives a positive constant $\gamma_1>0$ such that
\begin{equation}\label{Bg1}
    0 < \gamma(v, m) \leq \gamma_1.
\end{equation}
However, the condition $0 < \gamma(v, m)$ does not exclude the possible of degeneracy for the diffusion of $u$. Noting that $\gamma_v(v,m)\leq 0$ and hence to prevent degeneracy, the key point is to prove that $v$ admits an upper bound. In fact, if $\theta_1=0$, the following result provides the lower and upper bound for $v$.
\begin{lemma}\label{Bv1}
    If $\theta_1=0$, then the solution $(u,v,m)$ of \eqref{sys-1} obtained in Lemma \ref{LS} satisfies
    \begin{equation}\label{Bv1-1}
        0 \leq v(x, t) \leq C_1, \quad (x, t) \in \bar{\Omega} \times [0, T_{\max}),
    \end{equation}
    where $C_1>0$ is a constant independent of $t$.
\end{lemma}

\begin{proof}
    The fact $v \geq 0$ has been proved by Lemma \ref{LS}. We only need to show the upper bound of $v$.
    Multiplying the third equation of \eqref{sys-1} by $\theta_{2}$ and adding the result to the second equation of \eqref{sys-1}, one has
    \begin{equation*}
        (v + \theta_{2} m)_{t} = \Delta (v + \theta_{2} m) - v, \quad (x, t) \in \Omega \times (0, T_{\max}).
    \end{equation*}
    Letting $z= v+ \theta_{2} m$, we deduce that $z$ satisfies the following system 
    \begin{equation}\label{Bv1-2}
        \begin{cases}
            z_{t} - \Delta z + z = \theta_{2} m, \quad & x \in \Omega, \ t \in (0,T_{\max}), \\
            \nabla z \cdot \nu = 0, \quad & x \in \partial\Omega, \ t \in (0,T_{\max}), \\
            z(x, 0) = (v + \theta_{2} m) (x, 0) = v_{0}(x) + \theta_{2} m_{0}(x), \quad & x \in \Omega.
        \end{cases}
    \end{equation}
    Then applying the variation-of-constants formula, from \eqref{Bv1-2} we obtain 
    \begin{equation*}
        z(x, t)=e^{(\Delta - 1) t} z_{0}(x)+\int_{0}^{t} e^{(\Delta - 1) (t- s)} \theta_{2} m(x, s) d s,
    \end{equation*}
    which, together with the well-known semigroup estimate (\cite{Winkler-2010-JDE}) and the fact \eqref{L1-2}, gives
    \begin{equation}\label{esz}
        \begin{split}
            \|z(\cdot, t)\|_{L^{\infty}} & \leq \|e^{(\Delta - 1) t} z_{0}\|_{L^{\infty}} + \int_{0}^{t} \|e^{(\Delta - 1) (t- s)} \theta_{2} m(\cdot, s)\|_{L^{\infty}} d s \\
            & \leq k_{1} e^{- \lambda_{1} t} \|z_{0}\|_{L^{\infty}} + \theta_{2} k_{1} \|m_{0}\|_{L^{\infty}} \int_{0}^{t} e^{- \lambda_{1} (t - s)} d s \\
            & \leq k_{1} \left(1 + \frac{\theta_{2}}{\lambda_{1}}\right) \|z_{0}\|_{L^{\infty}}, \quad t \in [0, T_{\max}),
        \end{split}
    \end{equation}
    where $k_{1}$ is a positive constant depending only on $\Omega$ and $\lambda_{1} > 0$ is the first nonzero eigenvalue of $- \Delta$ in $\Omega$.
    Therefore, we derive from \eqref{esz} that
    \begin{equation*}
        \|v+\theta_2 m\|_{L^\infty}\leq c_1,
    \end{equation*}
which, combined with the fact $\theta_{2} > 0$ and the non-negativity of $m$, implies $\|v(\cdot, t)\|_{L^{\infty}} \leq c_2$ for $t\in[0,T_{\max})$. Hence, we complete the proof of Lemma \ref{Bv1}.
\end{proof}

\subsection{Global boundedness: $\theta_1=0$}
In this subsection, we shall prove the global boundedness of solution for the system \eqref{sys-1} with $\theta_1=0$. Under the assumption (H1) and the facts \eqref{Bv1-1} and \eqref{L1-2}, it holds that
\begin{equation}\label{Bg1-1}
    0 < \gamma_{0} \leq \gamma (v, m) \leq \gamma_{1}, \quad |\gamma_{v} (v, m)| \leq \gamma_{2}, \quad |\gamma_{m} (v, m)| \leq \gamma_{3}.
\end{equation}
    
\begin{lemma}\label{Lut}
    Let $(u,v,w)$ be the solution obtained in Lemma \ref{LS}. 
    Then there exists a positive constant $C_2>0$ independent of $t$ such that
    \begin{equation}\label{Lut-1}
        \int_{t}^{t + \tau} \int_{\Omega} u^{2} \leq C_2, \quad \text{for all} \ t\in (0,\widetilde{T}_{\max}),
    \end{equation}
    where
    \begin{equation}\label{DT}
        \tau := \min \left\{1, \frac{1}{2} T_{\max}\right\} \quad \text{and} \quad \widetilde{T}_{\max} := T_{\max} - \tau.
    \end{equation}
\end{lemma}
    
\begin{proof}
    Let $\mathcal{B}$ be a self-adjoint realization of $- \Delta$ defined on
    \begin{equation*}
        D(\mathcal{B}) := \left\{ \phi \in W^{2,2} (\Omega) \cap L^{2} (\Omega) \Big| \int_{\Omega} \phi = 0  \ \text{and} \  \frac{\partial\phi}{\partial\nu} = 0 \ \text{on} \ \partial \Omega \right\}.
    \end{equation*}
    Let $\bar{f}(t) = \frac{1}{| \Omega |} \int_{\Omega} f(x, t) dx$, we infer from \eqref{sys-1} that
    \begin{equation}\label{Lut-2}
        (u + \alpha m - \bar{u} - \alpha \bar{m})_{t} = - \mathcal{B} \left(u \gamma (v, m) + \alpha m - \overline{u \gamma (v, m)} - \alpha \bar{m}\right).
    \end{equation}
    Multiplying \eqref{Lut-2} by $\mathcal{B}^{-1} (u + \alpha m - \bar{u} - \alpha \bar{m})$ and integrating the result by parts, we obtain
    \begin{equation*}
        \begin{split}
            \frac{1}{2} & \frac{d}{dt} \int_{\Omega} | \mathcal{B}^{ - \frac{1}{2} } (u + \alpha m - \bar{u} - \alpha \bar{m}) |^{2} \\
            = & - \int_{\Omega} \mathcal{B} \left(u \gamma (v, m) + \alpha m - \overline{u \gamma (v, m)} - \alpha \bar{m}\right) \mathcal{B}^{-1} (u + \alpha m - \bar{u} - \alpha \bar{m}) \\
            = & - \int_{\Omega} \left(u \gamma (v, m) + \alpha m - \overline{u \gamma (v, m)} - \alpha \bar{m}\right) (u + \alpha m - \bar{u} - \alpha \bar{m}) \\
            = & - \int_{\Omega} \gamma (v, m) (u - \bar{u})^{2} - \alpha^{2} \int_{\Omega} (m - \bar{m})^{2} - \alpha \int_{\Omega} [1 + \gamma (v, m)] (u - \bar{u}) (m - \bar{m}) \\
            & - \bar{u} \int_{\Omega} \gamma (v, m) (u - \bar{u}) - \alpha \bar{u} \int_{\Omega} \gamma (v, m) (m - \bar{m}),
        \end{split}
    \end{equation*}
    which, together with the facts \eqref{Bg1-1}, the non-negativity of $u$ and $m$ and Lemma \ref{L1}, gives
    \begin{equation}\label{Lut-3}
        \begin{split}
            & \frac{1}{2} \frac{d}{dt} \int_{\Omega} | \mathcal{B}^{ - \frac{1}{2} } (u + \alpha m - \bar{u} - \alpha \bar{m})|^{2} + \gamma_{0} \int_{\Omega} (u - \bar{u})^{2} + \alpha^{2} \int_{\Omega} (m - \bar{m})^{2} \\
            & \leq - \alpha \int_{\Omega} [1 + \gamma (v, m)] (u - \bar{u}) (m - \bar{m}) - \bar{u} \int_{\Omega} \gamma (v, m) (u - \bar{u}) - \alpha \bar{u} \int_{\Omega} \gamma (v, m) (m - \bar{m}) \\
            & \leq \alpha \bar{u} \int_{\Omega} [1 + \gamma (v, m)] m + \alpha \bar{m} \int_{\Omega} [1 + \gamma (v, m)] u + (\bar{u}^{2} + \alpha \bar{u} \bar{m}) \int_{\Omega} \gamma (v, m) \\
            & \leq \frac{2 \alpha + 3 \alpha \gamma_{1}}{|\Omega|} \|u\|_{L^{1}} \|m\|_{L^{1}} + \frac{\gamma_{1}}{|\Omega|} \|u\|_{L^{1}}^{2}\\
            &\leq c_1.
        \end{split}
    \end{equation}
    By observing $\int_{\Omega} \mathcal{B}^{-\frac{1}{2}} (u + \alpha m - \bar{u} - \alpha \bar{m}) = 0$ and invoking the Poincar$\acute{\text{e}}$ inequality and the property $\|m(\cdot,t)\|_{L^{\infty}} \leq \|m_{0}\|_{L^{\infty}}$, we derive
    \begin{equation}\label{Lut-5}
        \begin{split}
            \int_{\Omega} |\mathcal{B}^{- \frac{1}{2}} (u + \alpha m - \bar{u} - \alpha \bar{m})|^{2} 
            & \leq c_{2} \int_{\Omega} |\nabla \mathcal{B}^{- \frac{1}{2}} (u + \alpha m - \bar{u} - \alpha \bar{m}) |^{2} \\
            & = c_{2} \int_{\Omega} |u + \alpha m - \bar{u} - \alpha \bar{m}|^{2} \\
            & \leq 2 c_{2} \int_{\Omega} (u - \bar{u})^{2} + 2 \alpha^{2} c_{2} |\Omega| \|m_{0}\|_{L^{\infty}}^{2}.
        \end{split}
    \end{equation}
    Multiplying \eqref{Lut-5} by $\frac{\gamma_0}{2c_2}$ and adding it into \eqref{Lut-3}, one has 
    \begin{equation}\label{Lut-6}
        X^{'} (t) + \frac{\gamma_{0}}{2 c_{2}} X (t) + \gamma_{0} \int_{\Omega} (u - \bar{u})^{2} \leq c_3,
    \end{equation}
    where $$X (t) = \int_{\Omega} |\mathcal{B}^{- \frac{1}{2}} (u + \alpha m - \bar{u} - \alpha \bar{m})|^{2}.$$
    Then applying the Gr\"{o}nwall's inequality to \eqref{Lut-6}, we first establish
    \begin{equation}\label{Lut-6*}
        X(t) = \int_{\Omega} |\mathcal{B}^{- \frac{1}{2}} (u + \alpha m - \bar{u} - \alpha \bar{m})|^{2} \leq c_4.
    \end{equation}
    Therefore, integrating \eqref{Lut-6} over $(t, t + \tau)$ with $\tau := \min \left\{1, \frac{1}{2} T_{\max}\right\}$ and employing \eqref{Lut-6*}, it follows that
    \begin{equation*}
        \int_{t}^{t + \tau} \int_{\Omega} (u - \bar{u})^{2} \leq \frac{c_{3} \tau + c_{4}}{\gamma_{0}} \leq \frac{c_{3} + c_{4}}{\gamma_{0}},
        \end{equation*}
    which, combined with $\int_{\Omega} (u - \bar{u})^{2} = \int_{\Omega} u^{2} - \int_{\Omega} \bar{u}^{2}$, yields
    \begin{equation*}
        \int_{t}^{t + \tau} \int_{\Omega} u^{2} =  \int_{t}^{t + \tau} \int_{\Omega} (u - \bar{u})^{2} + \int_{t}^{t + \tau} \int_{\Omega} \bar{u}^{2} \leq\frac{c_{3} + c_{4}}{\gamma_{0}} + \bar{u}^{2} |\Omega| \tau,
    \end{equation*}
    and hence \eqref{Lut-1} holds by virtue of $\bar{u} \leq \|u_{0}\|_{L^{\infty}} + \alpha \|m_{0}\|_{L^{\infty}}$.
\end{proof}
    
\begin{lemma}\label{Lvm}
    Let the conditions in Lemma \ref{Lut} hold. Then the solution $(u,v,m)$ of system \eqref{sys-1} with $\theta_1=0$ satisfies 
    \begin{equation}\label{Lvm-1}
        \int_{\Omega} |\nabla v|^{2}+\int_\Omega |\nabla m|^2 \leq C_3, \quad \text{for all} \ t \in (0, T_{\max}),
    \end{equation}
    and 
    \begin{equation}\label{Lvm-2}
        \int_{t}^{t + \tau} \int_{\Omega} |\Delta v|^{2} +\int_{t}^{t + \tau} \int_{\Omega} |\Delta m|^{2}\leq C_4, \quad \text{for all} \ t \in (0, \tilde{T}_{\max}),
    \end{equation}
    where $C_3$ and $C_4$ are positive constants independent of $t$.
\end{lemma}
    
\begin{proof}
    We multiply the second equation of \eqref{sys-1} by $- \Delta v$ and integrate the result by parts to obtain 
    \begin{equation*}
        \begin{split}
            \frac{1}{2} \frac{d}{dt} \int_{\Omega} |\nabla v|^{2} & = - \int_{\Omega} |\Delta v|^{2} + \int_{\Omega} v \Delta v - \theta_{2} \int_{\Omega} u m \Delta v \\
            & \leq - \frac{1}{2} \int_{\Omega} |\Delta v|^{2} - \int_{\Omega} |\nabla v|^{2} + \frac{\theta_{2}^2 \|m_0\|_{L^\infty}^2}{2} \int_{\Omega} u^{2},
        \end{split}
    \end{equation*}
    which leads to 
    \begin{equation}\label{Lvm-3}
        \frac{d}{dt} \int_{\Omega} |\nabla v|^{2} +2 \int_{\Omega} |\nabla v|^{2}+ \int_{\Omega} |\Delta v|^{2} \leq \theta_{2}^2 \|m_0\|_{L^\infty}^2 \int_{\Omega} u^{2}.
    \end{equation}
    Then applying Lemma \ref{LB3} and \eqref{Lut-1}, we derive from \eqref{Lvm-3} that
    \begin{equation}\label{Lvm-4}
        \int_{\Omega} |\nabla v|^{2}\leq c_1.
    \end{equation}
    On the other hand, integrating \eqref{Lvm-3} over $(t, t + \tau)$ for $t \in (0, \tilde{T}_{\max})$ and utilizing \eqref{Lvm-4}, we deduce that
    \begin{equation}\label{Lvm-5}
        \int_{t}^{t + \tau} \int_{\Omega} |\Delta v|^{2} \leq \theta_{2}^2 \|m_{0}\|_{L^{\infty}}^{2} \int_{t}^{t + \tau} \int_{\Omega} u^{2} + \int_{\Omega} |\nabla v|^{2}\leq c_2.
    \end{equation}
    Similarly, we rewrite the third equation of \eqref{sys-1} as follows
    \begin{equation}\label{Lvm-6}
        m_{t} = \Delta m - m + (1 - u) m.
    \end{equation}
    Thus, multiplying \eqref{Lvm-6} by $-\Delta m$ and integrating it by parts, we end up with 
    \begin{equation*}
        \begin{split}
            \frac{1}{2}\frac{d}{dt} \int_\Omega |\nabla m|^2+\int_\Omega |\nabla m|^2 + \int_\Omega |\Delta m|^2
           & = - \int_\Omega (1-u) m\Delta m\\
           &\leq \frac{1}{2}\int_\Omega |\Delta m|^2 + \|m_0\|^2_{L^\infty} |\Omega| + \|m_0\|_{L^\infty}^2\int_\Omega u^2,
        \end{split}
    \end{equation*}
    which yields
    \begin{equation}\label{Lvm-7}
        \begin{split}
            \frac{d}{dt}\int_\Omega |\nabla m|^2 + 2 \int_\Omega |\nabla m|^2 + \int_\Omega |\Delta m|^2
            &\leq 2\|m_0\|^2_{L^\infty} |\Omega| + 2 \|m_0\|_{L^\infty}^2 \int_\Omega u^2.
        \end{split}
    \end{equation}
    Therefore, employing Lemma \ref{LB3} and \eqref{Lut-1} again, from \eqref{Lvm-7} we can find two positive constants $c_3$ and $c_4$ such that 
    \begin{equation}\label{Lvm-8}
        \int_\Omega |\nabla m|^2 \leq c_3
    \end{equation}
    and
    \begin{equation}\label{Lvm-9}
        \int_t^{t+\tau} \int_\Omega |\Delta m|^2 \leq c_4.
    \end{equation}
    Hence, the combination of \eqref{Lvm-4} and \eqref{Lvm-8} gives \eqref{Lvm-1}. The estimate \eqref{Lvm-2} is a direct consequence of \eqref{Lvm-5} and \eqref{Lvm-9}.
\end{proof}

Next, we shall establish the boundedness of $\|u(\cdot,t)\|_{L^2}$ as follows.
\begin{lemma}\label{Lu2}
    Let $(u, v, m)$ be a solution of the system \eqref{sys-1} obtained in Lemma \ref{LS}. 
    Then there exists a positive constant $C_5$ independent of $t$ such that 
    \begin{equation}\label{Lu2-1}
        \|u(\cdot, t)\|_{L^{2}} \leq C_5, \quad \text{for all} \ t \in (0, T_{\max}).
    \end{equation}
\end{lemma}

\begin{proof}
    Multiplying the first equation of \eqref{sys-1} by $u$ and integrating the result by parts, then invoking  $0 < \gamma_{0} \leq \gamma (v, m) \leq \gamma_{1}, \ |\gamma_{v} (v, m)| \leq \gamma_{2}, \ |\gamma_{m} (v, m)| \leq \gamma_{3}$ in \eqref{Bg1-1}, we deduce that
    \begin{equation*}
         \begin{split}
            \frac{1}{2} \frac{d}{d t} \int_{\Omega} u^{2}  = &- \int_{\Omega} \nabla u \cdot \nabla(\gamma (v, m) u) + \alpha \int_{\Omega} u^{2} m \\
            \leq &- \int_{\Omega} \gamma(v, m) |\nabla u|^{2} - \int_{\Omega} \gamma_{v} (v, m) u \nabla u \cdot \nabla v \\
            &- \int_{\Omega} \gamma_{m} (v, m) u \nabla u \cdot \nabla m + \alpha \int_{\Omega} u^{2} m \\
            & \leq - \frac{\gamma_{0}}{2} \int_{\Omega} |\nabla u|^{2} + \frac{\gamma_{2}^{2}}{\gamma_{0}}  \int_{\Omega} u^{2}|\nabla v|^{2} +\frac{\gamma_{3}^{2}}{\gamma_0} \int_\Omega u^2|\nabla m|^{2} + \alpha \int_{\Omega} u^{2} m,
        \end{split}
    \end{equation*}
    which, combined  with $m(x,t)\leq \|m_0\|_{L^\infty}$ in \eqref{L1-2}, gives
    \begin{equation}\label{Lu2-2}
        \frac{d}{d t} \int_{\Omega} u^{2} + \gamma_{0} \int_{\Omega} |\nabla u|^{2} \leq \frac{2 (\gamma_{2}^{2} + \gamma_{3}^{2})}{\gamma_{0}} \int_{\Omega} u^{2} (|\nabla v|^2 + |\nabla m|^2) + 2 \alpha \|m_0\|_{L^\infty} \int_{\Omega} u^{2}.
    \end{equation}
    Moreover, applying Gagliardo-Nirenberg inequality and Young inequality, and using the fact $\|\nabla v(\cdot,t)\|_{L^2}+\|\nabla m(\cdot,t)\|_{L^2}\leq c_1$ in \eqref{Lvm-1}, we derive
    \begin{equation*}
        \begin{split}
            \frac{2 (\gamma_{2}^{2} + \gamma_{3}^{2})}{\gamma_{0}} \int_{\Omega} u^{2} (|\nabla v|^2 + |\nabla m|^2) 
            &\leq\frac{2 (\gamma_{2}^{2} + \gamma_{3}^{2})}{\gamma_{0}} \|u\|_{L^{4}}^{2}(\|\nabla v\|_{L^{4}}^{2} + \|\nabla m\|_{L^{4}}^{2}) \\
            &\leq c_2 (\|\nabla u\|_{L^{2}}\|u\|_{L^{2}}+\|u\|_{L^{2}}^{2}) (\|\Delta v\|_{L^{2}}\|\nabla v\|_{L^{2}}+\|\nabla v\|_{L^{2}}^{2}) \\ 
            &\quad+ c_2 (\|\nabla u\|_{L^{2}} \|u\|_{L^{2}} + \|u\|_{L^{2}}^{2}) (\|\Delta m\|_{L^{2}} \|\nabla m\|_{L^{2}} + \|\nabla m\|_{L^{2}}^{2}) \\
            &\leq \gamma_{0} \|\nabla u\|_{L^{2}}^{2} + c_3 \|u\|_{L^{2}}^{2}(\|\Delta v\|_{L^{2}}^{2} +\|\Delta m\|_{L^2}^2+1),\\
        \end{split}
    \end{equation*}
    which, substituted into \eqref{Lu2-2}, yields
    \begin{equation}\label{Lu2-3}
        \frac{d}{d t} \|u\|_{L^{2}}^{2} \leq c_4 (1+ \|\Delta v\|_{L^{2}}^{2} + \|\Delta m\|_{L^{2}}^{2}) \|u\|_{L^{2}}^{2}.
    \end{equation}
    On the other hand, for any $t \in (0, T_{\max})$ and in the case of either $t \in (0, \tau)$ or $t \geq \tau$ with $\tau = \min \left\{1, \frac{T_{\max}}{2}\right\}$, from \eqref{Lut-1} we can find a $t_{0} = t_{0} (t) \in ((t - \tau)_{+}, t)$ satisfying $t_{0} \geq 0$ and $t_{0} \in (0, \widetilde{T}_{\max})$, such that
    \begin{equation}\label{Lu2-4}
        \|u(\cdot, t_{0})\|_{L^{2}}^{2} \leq c_5.
    \end{equation}
    Therefore, we integrate \eqref{Lu2-3} over $(t_{0}, t)$, and utilize the facts \eqref{Lu2-4}, \eqref{Lvm-2} and $t \leq t_{0} + \tau\leq t_0 +1$ to obtain
    \begin{equation*}
        \|u(\cdot, t)\|_{L^{2}}^{2} \leq \|u(\cdot, t_{0})\|_{L^{2}}^{2} e^{c_4 \int_{t_{0}}^{t} \left(1 + \|\Delta v(\cdot, s)\|_{L^{2}}^{2} + \|\Delta m(\cdot, s)\|_{L^{2}}^{2}\right) d s} \leq c_6,
    \end{equation*}
    which guarantees \eqref{Lu2-1} and then the proof of Lemma \ref{Lu2} is completed.
\end{proof}
	
\begin{lemma}\label{Lu4}
    Let $(u, v, m)$ be the solution of the system \eqref{sys-1} with $\theta_1=0$ obtained in Lemma \ref{LS}. Then it holds that
    \begin{equation}\label{Lu4-1}
        \|u(\cdot, t)\|_{L^{4}} \leq C_6, \quad \text{for all} \ t \in (0, T_{\max}),
    \end{equation}
    where $C_{6} > 0$ is a constant independent of $t$.
\end{lemma}
    
\begin{proof}
    We first show that there exists a constant $c_1>0$ such that 
    \begin{equation}\label{Lu4-2}
        \|\nabla v(\cdot,t)\|_{L^4}+\|\nabla m(\cdot,t)\|_{L^4}\leq c_1.
    \end{equation}
    In fact, an application of the variation-of-constants formula to the second equation of \eqref{sys-1} gives
    \begin{equation}\label{Lu4-3}
         v(\cdot, t) = e^{(\Delta-1)t} v_{0} + \theta_2 \int_{0}^{t} e^{(\Delta-1)(t - s) }  m(\cdot, s) u(\cdot, s) d s.
    \end{equation}       
    Then employing the well-known semigroup estimates \cite{Winkler-2010-JDE} and the fact $\|u(\cdot,t)\|_{L^2}\leq c_2$ in \eqref{Lu2-1}, we derive from \eqref{Lu4-3} that
    \begin{equation}\label{Lu4-4}
        \begin{split}
            \|\nabla v(\cdot, t) \Vert_{L^{4}} & \leq \Vert \nabla e^{(\Delta-1)t} v_{0} \Vert_{L^{4}} + \theta_{2}\int_{0}^{t} \Vert \nabla e^{  (\Delta-1)(t - s)}  m(\cdot, s) u(\cdot, s) \Vert_{L^{4}} d s \\
            &\leq  k_{1} e^{-\lambda_1 t} \Vert \nabla v_{0} \Vert_{L^{4}} + k_{2} c_2 \theta_{2} \|m_0\|_{L^\infty} \int_{0}^{\infty} \left(1 + (t - s)^{- \frac{3}{4}}\right) e^{- (\lambda_{1}+1) (t - s)} d s\\
            &\leq c_3,
        \end{split}
    \end{equation}
   where $k_1,k_2$ are positive constants depending only on $\Omega$ and $\lambda_1>0$ is the first nonzero eigenvalue of $-\Delta$ in $\Omega$ under Neumann boundary conditions.
   
   Similarly, we can apply the variation-of-constants formula to the third equation of \eqref{sys-1} such that
   \begin{equation}\label{Lu4-5}
        m(\cdot, t) = e^{ \Delta t } m_{0} - \int_{0}^{t} e^{\Delta (t - s) } u(\cdot, s) m(\cdot, s) d s.
    \end{equation}
   Therefore, we infer from \eqref{Lu4-5} that
    \begin{equation*}
        \begin{split}
            \|\nabla m(\cdot,t)\|_{L^{4}} &\leq \|\nabla e^{\Delta t} m_{0}\|_{L^{4}} + \int_{0}^{t} \|\nabla e^{\Delta(t - s)} m(\cdot, s) u(\cdot, s)\|_{L^{4}} d s \\
            & \leq k_{1} e^{- \lambda_{1} t} \|\nabla m_{0}\|_{L^{4}} + k_{2} c_2 \|m_0\|_{L^\infty} \int_{0}^{\infty} \left(1 + (t - s)^{- \frac{3}{4}}\right) e^{- \lambda_{1} (t - s)} d s \\
            &\leq c_4,
        \end{split}
    \end{equation*}
    which, together with \eqref{Lu4-4}, yields \eqref{Lu4-2}.

    With \eqref{Lu4-2} in hand, we proceed to prove \eqref{Lu4-1}. To this end, we multiply the first equation of \eqref{sys-1} by $u^{3}$, integrate the result by parts and use the fact \eqref{Bg1-1} to obtain
    \begin{equation*}
        \begin{split}
            \frac{1}{4} \frac{d}{d t} \int_{\Omega} u^{4} & = - 3 \int_{\Omega} u^{2} \nabla u \cdot \nabla(\gamma(v, m) u) + \alpha \int_{\Omega} u^{4} m \\
            & \leq - \frac{3\gamma_{0}}{2}  \int_{\Omega} u^{2} |\nabla u|^{2} + \frac{3 (\gamma_{2}^{2} + \gamma_{3}^{2})}{\gamma_{0}} \int_{\Omega} u^{4} \left(|\nabla v|^{2}+ |\nabla m|^{2}\right) + \alpha \int_{\Omega} u^{4} m,
        \end{split}
    \end{equation*}
    which together with $m\leq \|m_0\|_{L^\infty}$ in \eqref{L1-2}, gives
    \begin{equation}\label{Lu4-7}
        \frac{d}{d t} \int_{\Omega} u^{4} + \frac{3\gamma_{0}}{2} \int_{\Omega} |\nabla u^{2}|^{2} \leq \frac{12 (\gamma_{2}^{2} + \gamma_{3}^{2})}{\gamma_{0}} \int_{\Omega} u^{4} (|\nabla v|^{2}+ |\nabla m|^{2}) + 4 \alpha \|m_0\|_{L^\infty} \int_{\Omega} u^{4}.
    \end{equation}
    Invoking Gagliardo-Nirenberg inequality and Young's inequality, and noting \eqref{Lu2-1} and \eqref{Lu4-2}, one has
    \begin{equation}\label{Lu4-8}
        \begin{split}
            \frac{12 (\gamma_{2}^{2} + \gamma_{3}^{2})}{\gamma_{0}} \int_{\Omega} u^{4} (|\nabla v|^{2}+ |\nabla m|^{2}) 
            \leq & \frac{12 (\gamma_{2}^{2} + \gamma_{3}^{2})}{\gamma_{0}} \|u^{2}\|_{L^{4}}^{2} (\|\nabla v\|_{L^{4}}^{2} + \|\nabla m\|_{L^{4}}^{2}) \\
            \leq & \frac{24 (\gamma_{2}^{2} + \gamma_{3}^{2})c_1^2}{\gamma_{0}} \|u^{2}\|_{L^{4}}^{2} \\
            \leq & c_5 (\|\nabla u^{2}\|_{L^{2}}^{\frac{3}{2}} \|u^{2}\|_{L^{1}}^{\frac{1}{2}} + \|u^{2}\|_{L^{1}}^{2}) \\
            \leq & \gamma_{0} \|\nabla u^{2}\|_{L^{2}}^{2} +c_6.
        \end{split}
    \end{equation}
    Furthermore, employing the Gagliardo-Nirenberg inequality and Young's inequality again, we can find a constant $c_{7} > 0$ such that
    \begin{equation}\label{Lu4-9}
        \begin{split}
            (1 + 4 \alpha \|m_0\|_{L^\infty}) \int_{\Omega} u^{4} & = (1 + 4 \alpha \|m_0\|_{L^\infty}) \|u^{2}\|_{L^{2}}^{2} \\
            & \leq c_{7} (1 + 4 \alpha \|m_0\|_{L^\infty}) (\|\nabla u^{2}\|_{L^{2}}^{\frac{3}{2}} \|u^{2}\|_{L^{\frac{1}{2}}}^{\frac{1}{2}} + \|u^{2}\|_{L^{\frac{1}{2}}}^{2}) \\
            & \leq c_{7} (1 + 4 \alpha \|m_0\|_{L^\infty}) (\|\nabla u^{2}\|_{L^{2}}^{\frac{3}{2}} \|u\|_{L^{1}} + \|u\|_{L^{1}}^{4})\\
            & \leq \frac{\gamma_{0}}{2} \|\nabla u^{2}\|_{L^{2}}^{2} + c_8.
        \end{split}
    \end{equation}
    Thus, we substitute \eqref{Lu4-8} and \eqref{Lu4-9} into \eqref{Lu4-7} such that
    \begin{equation*}
         \frac{d}{d t} \int_{\Omega} u^{4} + \int_{\Omega} u^{4} \leq c_6+c_8,
    \end{equation*}
    which, along with the Gr\"{o}nwall's inequality, yields \eqref{Lu4-1}. This completes the proof.
\end{proof}

\begin{lemma}\label{LI}
    Suppose that the conditions in Lemma \ref{Lu4} hold. Let $(u, v, m)$ be the solution of the system \eqref{sys-1} with $\theta_1=0$. Then it follows that
    \begin{equation}\label{LI-1}
        \|u(\cdot, t)\|_{L^{\infty}} \leq C_7 \quad \text{for all} \ t \in (0, T_{\max}),
    \end{equation}
    where the constant $C_{7} > 0$ is independent of $t$.
\end{lemma}
    
\begin{proof}
    Invoking the semigroup estimates \cite{Winkler-2010-JDE} and noting the fact $\|u(\cdot,t)\|_{L^4}\leq c_1$ in \eqref{Lu4-1}, we infer from \eqref{Lu4-3} and \eqref{Lu4-5} that 
    \begin{equation*}
        \begin{split}
            \|\nabla v(\cdot, t)\|_{L^{\infty}} & \leq \|\nabla e^{(\Delta-1)t } v_{0}\|_{L^{\infty}} + \theta_{2}\int_{0}^{t} \|\nabla e^{(\Delta-1)(t - s) } m(\cdot, s) u(\cdot, s)\|_{L^{\infty}} d s \\
            & \leq c_{2} \|v_{0}\|_{W^{1, \infty}} + k_{2} \theta_{2} \|m_0\|_{L^\infty} \int_{0}^{t} \left(1 + (t - s)^{- \frac{3}{4}}\right) e^{- (\lambda_{1}+1) (t - s)} \|u(\cdot, s)\|_{L^{4}} d s \\
            & \leq c_{2} \|v_{0}\|_{W^{1, \infty}} + c_1 k_2 \theta_2 \|m_0\|_{L^\infty} \int_{0}^{\infty} \left(1 + (t - s)^{- \frac{3}{4}}\right) e^{- (\lambda_{1}+1) (t - s)}d s\\
            &\leq c_3
        \end{split}
    \end{equation*}
    and 
    \begin{equation*}
        \begin{split}
            \|\nabla m(\cdot, t)\|_{L^{\infty}} & \leq \|\nabla e^{\Delta t} m_{0}\|_{L^{\infty}} + \int_{0}^{t} \|\nabla e^{\Delta (t - s) }  u(\cdot, s) m(\cdot, s)\|_{L^{\infty}} d s \\
            & \leq c_{4} \|m_{0}\|_{W^{1, \infty}} + k_{2} c_1 \|m_0\|_{L^\infty} \int_{0}^{t} \left(1 + (t - s)^{- \frac{3}{4}}\right) e^{- \lambda_{1} (t - s)}  d s \\
            & \leq c_5,
        \end{split} 
    \end{equation*}
    and hence there exists a constant $c_6>0$ such that
    \begin{equation}\label{LI-2}
        \|\nabla v(\cdot,t)\|_{L^\infty}+\|\nabla m(\cdot,t)\|_{L^\infty}\leq c_6.
    \end{equation}
   With \eqref{LI-2} and \eqref{Bg1-1} in hand, 
    we multiply the first equation of \eqref{sys-1} by $u^{p - 1} (p \geq 2)$ and integrate the result by parts such that
    \begin{equation}\label{LI-3}
        \begin{split}
            \frac{1}{p} \frac{d}{d t} \int_{\Omega} u^{p} = & - (p - 1) \int_{\Omega} u^{p - 2} \nabla u \cdot \nabla (\gamma (v, m) u) + \alpha \int_{\Omega} u^{p} m \\
            \leq & - \frac{\gamma_{0} (p - 1)}{2} \int_{\Omega} u^{p - 2} |\nabla u|^{2} + \frac{(p - 1)\gamma_2^2}{\gamma_0} \int_{\Omega} u^{p} |\nabla v|^{2} \\
            &+ \frac{(p - 1)\gamma_3^2}{\gamma_0} \int_{\Omega} u^{p} |\nabla m|^{2}  + \alpha \int_{\Omega} u^{p} m \\
            \leq & - \frac{\gamma_{0} (p - 1)}{2} \int_{\Omega} u^{p - 2} |\nabla u|^{2} +c_7 (p - 1) \int_{\Omega}u^{p},
        \end{split}
    \end{equation}
    where $c_7:=\frac{(\gamma_{2}^{2} + \gamma_{3}^{2}) c_6^2}{\gamma_{0}} + \alpha \|m_0\|_{L^\infty}>0$ is independent of $p$.
    Employing the identity $\int_{\Omega} u^{p - 2} |\nabla u|^{2} = \frac{4}{p^{2}} \int_{\Omega} |\nabla u^{\frac{p}{2}}|^{2}$, we can obtain from \eqref{LI-3} that
    \begin{equation}\label{LI-4}
        \frac{d}{d t} \int_{\Omega} u^{p} + p (p - 1) \int_{\Omega} u^{p} \leq - \frac{2 \gamma_{0} (p - 1)}{p} \int_{\Omega} |\nabla u^{\frac{p}{2}}|^{2} + (c_7 + 1) p (p - 1) \int_{\Omega}u^{p}.
    \end{equation}
    We recall the inequality $\|f\|_{L^{2}}^{2} \leq \epsilon \|\nabla f\|_{L^{2}}^{2} + c (1 + \epsilon^{- 1}) \|f\|_{L^{1}}^{2}$ for any $\epsilon > 0$, where $c > 0$ only depends on $\Omega$ (see \cite{JSW-2020-JDE}). Then letting $f = u^{\frac{p}{2}}$ and $\epsilon=\frac{2\gamma_0}{(c_7+1)p^2}$, we deduce that
    \begin{equation}\label{LI-5}
        (c_7 + 1) p (p - 1) \int_{\Omega}u^{p} \leq \frac{2 \gamma_{0} (p - 1)}{p} \int_{\Omega} |\nabla u^{\frac{p}{2}}|^{2} +c_8 p (p - 1) (1 + p^{2}) \left(\int_{\Omega} u^{\frac{p}{2}}\right)^{2},
    \end{equation}
    where $c_8>0$ is a constant independent of $p$.
    From $1 + p^{2} \leq (1 + p)^{2}$, substituting \eqref{LI-5} into \eqref{LI-4} implies
    \begin{equation*}
        \frac{d}{d t} \int_{\Omega} u^{p} + p (p - 1) \int_{\Omega} u^{p} \leq c_8 p (p - 1) (1 + p)^{2} \left(\int_{\Omega} u^{\frac{p}{2}}\right)^{2},
    \end{equation*}
    which gives
    \begin{equation}\label{LI-6}
        \int_{\Omega} u^{p} (x, t) \leq \int_{\Omega} u_{0}^{p} (x) + c_8 (1 + p)^{2} \sup_{0 \leq t \leq T_{\max}} \left(\int_{\Omega} u^{\frac{p}{2} }(x, t)\right)^{2}.
    \end{equation}
    Hence, applying the Moser iteration \cite{Alikakos-1979-CPDE} (see also the similar argument as in \cite{Tao-2011-JMAA}), from \eqref{LI-6} we derive
    \begin{equation*}
        \|u(\cdot, t)\|_{L^{\infty}} \leq 2^{6} c_8 (1 + |\Omega|) (1 + \alpha) (\|u_{0}\|_{L^{\infty}} + \|m_{0}\|_{L^{\infty}}),
    \end{equation*}
    which yields \eqref{LI-1}.
\end{proof}
    
\begin{proof}[\textbf{Proof of Theorem \ref{GBS-1}: Boundedness}]
    For any fixed $\alpha \geq 0$ and $n=2$, from Lemma \ref{LI}, we can find a constant $C > 0$ independent of $t$ such that
    \begin{equation*}
        \|u(\cdot, t)\|_{L^{\infty}} \leq C ,
    \end{equation*}
    which, combined with the local existence results in Lemma \ref{LS}, proves the existence of global classical solution with uniform-in-time bound as stated in Theorem \ref{GBS-1}.
\end{proof}

\subsection{Global stabilization: $\theta_1=0$}
In this subsection, we shall show the global stabilization of solution for the system \eqref{sys-1} with $\theta_1=0$. For this purpose, we first improve the regularities as follows.
\begin{lemma}\label{Rg}
    Let $(u, v, m)$ be the non-negative global classical solution of \eqref{sys-1} obtained in Theorem \ref{GBS-1}. Then there exist $\sigma \in (0,1)$ and $C>0$ such that
    \begin{equation*}
        \|(u,v,m)\|_{C^{2 + \sigma, 1 + \frac{\sigma}{2}}(\bar{\Omega} \times[t, t + 1])} \leq C, \quad \text {for all} \ t \geq 1.
    \end{equation*}
\end{lemma}

\begin{proof} 
    We can use the similar arguments as in \cite[Lemma 4.1]{JSW-2020-JDE} to prove Lemma \ref{Rg}. We omit the details of proof for convenience.
\end{proof}

\begin{lemma}\label{Lm}
    Let $(u, v, m)$ be the solution of the system \eqref{sys-1} obtained in Theorem \ref{GBS-1}. Then it holds that
    \begin{equation}\label{Lm-1}
        \int_{0}^{\infty} \int_{\Omega} um < \infty
    \end{equation}
    and
    \begin{equation}\label{Lm-2}
        \int_{0}^{\infty} \int_{\Omega} \left| \nabla m \right|^{2} < \infty.
    \end{equation}
\end{lemma}

\begin{proof}
    We derive from the third equation of \eqref{sys-1} and the homogeneous Neumann boundary condition that
    \begin{equation*}
        \frac{d}{d t} \int_{\Omega} m + \int_{\Omega} u m = 0,
    \end{equation*}
    which gives
    \begin{equation}\label{Lm-3}
        \int_{0}^{t} \int_{\Omega} u m \leq \int_{\Omega} m_{0}, \quad \text{for all} \ t > 0,
    \end{equation}
    and \eqref{Lm-1} is an immediate result of \eqref{Lm-3}.
     Moreover, multiplying the third equation of \eqref{sys-1} by $m$ and integrating it by parts, we obtain
    \begin{equation}\label{Lm-4}
        \frac{1}{2} \frac{d}{d t} \int_{\Omega} m^{2} = - \int_{\Omega} |\nabla m|^2 - \int_{\Omega} u m^2.
    \end{equation}
    By the direct integration with respect to time, we infer from \eqref{Lm-4} that 
    \begin{equation*}
        \int_{0}^{t} \int_{\Omega} |\nabla m|^{2} \leq \frac{1}{2} \int_{\Omega} m_{0}^{2},
    \end{equation*}
    which yields \eqref{Lm-2}.
\end{proof}

\begin{lemma}\label{Lm1}
    Let the conditions in Lemma \ref{Lm} hold. Then there exists a time sequence $(t_{k})_{k \in \mathbb{N}} \subset (0, \infty)$ satisfying $t_{k} \to \infty$ as $k \to \infty$ such that
    \begin{equation}\label{Lm1-1}
        \int_{t_{k}}^{t_{k} + 1} \int_{\Omega} m \to 0 \quad \text{as} \quad k \to \infty.
    \end{equation}
\end{lemma}

\begin{proof}
    From \eqref{Lm-1}, we derive 
    \begin{equation}\label{Lm1-2}
        \int_{j}^{j+1} \int_{\Omega} u m \to 0, \quad \text{as} \quad j \to \infty.
    \end{equation}
    In addition, defining $\bar{m} := \frac{1}{|\Omega|} \int_{\Omega} m$, we deduce that
    \begin{equation}\label{Lm1-3}
        \int_{j}^{j+1} \int_{\Omega} u m = \int_{j}^{j+1} \int_{\Omega} u (m - \bar{m}) + \int_{j}^{j+1} \int_{\Omega} u \bar{m} = I_{1} (j) + I_{2} (j).
    \end{equation}
    Owing to \eqref{Lm-2} and $\|u(\cdot, t)\|_{L^{2}} \leq c_1$ in \eqref{Lu2-1}, we apply the H{\"o}lder inequality and Poincar$\acute{\text{e}}$ inequality to obtain
    \begin{equation}\label{Lm1-4}
        \begin{split}
            |I_{1} (j)| & \leq \left(\int_{j}^{j+1} \int_{\Omega} u^{2} \right)^{\frac{1}{2}} \cdot \left(\int_{j}^{j+1} \int_{\Omega} |m - \bar{m}|^{2}\right)^{\frac{1}{2}} \\
            & \leq c_{2} \left(\int_{j}^{j+1} \int_{\Omega} |m-\bar{m}|^{2}\right)^{\frac{1}{2}} \\
            & \leq c_{3} \left(\int_{j}^{j+1} \int_{\Omega} |\nabla m|^{2}\right)^{\frac{1}{2}} \to 0, \quad \text{as} \quad j \to \infty.
        \end{split}
    \end{equation}
    Then the combination of \eqref{Lm1-2}, \eqref{Lm1-3} and \eqref{Lm1-4} gives  $I_{2} (j) \to 0$ as $j \to \infty$.
        
    Moreover, using $\|u_0\|_{L^1}\leq \|u(\cdot,t)\|_{L^1}$ in \eqref{L1-1} and the fact that $I_{2} (j) \to 0$ as $j \to \infty$, one has
    \begin{equation*}
        \bar{u}_0 \int_{j}^{j + 1} \int_{\Omega} m = \|u_{0}\|_{L^1} \int_{j}^{j+1} \bar{m} \leq I_{2} (j) = \int_{j}^{j+1} \int_{\Omega} u \bar{m} \to 0, \quad \text{as} \quad j \to \infty,
    \end{equation*}
    which yields
    \begin{equation}\label{Lm1-5}
        \int_{j}^{j + 1} \int_{\Omega} m \to 0, \quad \text{as} \quad j \to \infty.
    \end{equation}
    Therefore, we set $j = t_{k}$, $k \in \mathbb{N}$ such that $t_{k} \to \infty$ as $k \to \infty$ and then \eqref{Lm1-1} is a direct consequence of \eqref{Lm1-5}.
    Hence, the proof of Lemma \ref{Lm1} is completed.
\end{proof}

\begin{lemma}\label{Lmi}
    Let $(u, v, m)$ be the solution of the system \eqref{sys-1} obtained in Theorem \ref{GBS-1}. Then it follows that
    \begin{equation}\label{Lmi-1}
        \|m(\cdot, t)\|_{L^{\infty}} \to 0 \quad \text{as} \quad t \to \infty.
    \end{equation}
\end{lemma}
	
\begin{proof}
    We employ the Gagliardo-Nirenberg inequality to find a constant $c_{1} > 0$ satisfying
    \begin{equation}\label{Lmi-2}
        \begin{split}
            \|m(\cdot, t)\|_{L^\infty} & \leq c_{1} \|\nabla m(\cdot, t)\|_{L^4}^{\frac{4}{5}} \|m(\cdot, t)\|_{L^1}^{\frac{1}{5}} + c_{1} \|m(\cdot, t)\|_{L^1} \\
            & \leq \mu \|\nabla m(\cdot, t)\|_{L^4} + c_{2} \|m(\cdot, t)\|_{L^1},
        \end{split}
    \end{equation}
    where $\mu > 0$ is an arbitrary constant and $c_{2} > 0$ is a constant depending on $\mu$. Let $(t_{k})_{k \in \mathbb{N}} \subset (0, \infty)$ be the sequence chosen in Lemma \ref{Lm1}. Then an application of $\|\nabla m(\cdot,t)\|_{L^4}\leq c_3$ and \eqref{Lmi-2} shows 
    \begin{equation*}
        \int_{t_{k}}^{t_{k} + 1} \|m(\cdot, t)\|_{L^\infty}\leq \mu c_3+c_2\int_{t_{k}}^{t_{k} + 1} \|m(\cdot,t)\|_{L^1},
    \end{equation*}
    which, together with the arbitrary of $\mu$ and \eqref{Lm1-1}, gives
    \begin{equation*}
        \int_{t_{k}}^{t_{k} + 1} \|m(\cdot, t)\|_{L^\infty} \to 0, \quad \text{as} \quad k \to \infty,
    \end{equation*}
    and hence
    \begin{equation}\label{Lmi-3}
        \liminf_{t \to \infty} \|m(\cdot, t)\|_{L^\infty} = 0.
    \end{equation}
    Combining \eqref{Lmi-3} with the fact that $t \mapsto \|m(\cdot, t)\|_{L^\infty}$ is monotone as shown in Lemma \ref{L1}, we obtain \eqref{Lmi-1} and hence the proof of Lemma \ref{Lmi} is completed.
\end{proof}
\begin{lemma}\label{Lu2-c}	
    Let $(u, v, m)$ be the solution of the system \eqref{sys-1} obtained in Theorem \ref{GBS-1}. Then it holds that
    \begin{equation}\label{Lu2-c1}
        \lim_{t \to \infty} (\|u(\cdot, t) - u_{*}\|_{L^\infty} + \|v(\cdot, t)\|_{L^\infty}) = 0,
    \end{equation}
    where $u_{*} = \frac{1}{|\Omega|} \int_\Omega u_{0} + \frac{\alpha}{|\Omega|} \int_\Omega m_{0}.$
\end{lemma}

\begin{proof}
    We rewrite the first equation of the system \eqref{sys-1} as
    \begin{equation}\label{Lu2-c2}
        (u - u_{*})_{t} = \Delta (u \gamma (v, m)) + \alpha u m.
    \end{equation}
    Then multiplying \eqref{Lu2-c2} by $u - u_{*}$ and integrating it by parts,  we derive
    \begin{equation*}
        \begin{split}
            & \frac{1}{2} \frac{d}{d t} \int_{\Omega} (u-u_{*})^{2} + \int_{\Omega}\gamma(v, m) |\nabla u|^{2} \\
            & = - \int_{\Omega} \gamma_{v}(v,m) u \nabla u \cdot \nabla v - \int_{\Omega} \gamma_{m}(v,m) u \nabla u \cdot \nabla m + \alpha \int_{\Omega} u m (u-u_{*}), 
        \end{split}
    \end{equation*}
    which, combined with $0 < \gamma_{0} \leq \gamma (v, m) \leq \gamma_{1}$, $|\gamma_{v} (v, m)| \leq \gamma_{2}$, $|\gamma_{m} (v, m)| \leq \gamma_{3}$ in \eqref{Bg1-1} and $\|u(\cdot,t)\|_{L^\infty}\leq c_1$, yields
    \begin{equation*}
        \begin{split}
            &\frac{d}{d t} \int_{\Omega} (u - u_{*})^{2} + 2\gamma_{0} \int_{\Omega} |\nabla u|^{2} \\
            &\leq 2 c_1\gamma_2  \int_\Omega |\nabla u||\nabla v| + 2c_1\gamma_3 \int_\Omega |\nabla u||\nabla m|+ 2 c_1^2 \alpha \int_{\Omega} m\\
            &\leq \gamma_0 \int_\Omega|\nabla u|^2+\frac{2c_1^2\gamma_2^2}{\gamma_0} \int_\Omega |\nabla v|^2 + \frac{2c_1^2\gamma_3^2}{\gamma_0} \int_\Omega |\nabla m|^2 + 2 c_1^2\alpha \int_{\Omega} m,
        \end{split}
    \end{equation*}
    and hence 
    \begin{equation}\label{Lu2-c4}
        \frac{d}{d t} \int_{\Omega} (u - u_{*})^{2} + \gamma_{0} \int_{\Omega} |\nabla u|^{2} \leq \frac{2c_1^2\gamma_2^2}{\gamma_0} \int_\Omega |\nabla v|^2+\frac{2c_1^2\gamma_3^2}{\gamma_0} \int_\Omega |\nabla m|^2+2 c_1^2 \alpha \int_{\Omega} m.
    \end{equation}
    From the Poincar$\acute{\text{e}}$ inequality, there exists a positive constant $C_p$ such that
    \begin{equation*}
        \int_{\Omega} (u - \bar{u})^{2} \leq C_{p} \int_{\Omega} |\nabla u|^2,
    \end{equation*}
    which along with
    \begin{equation*}
        \bar{u} = \frac{1}{|\Omega|} \int_\Omega u = u_*- \frac{\alpha}{|\Omega|} \int_\Omega m = u_*- \alpha \bar{m},
    \end{equation*}
    gives
    \begin{equation*}
        \int_{\Omega} (u - u_{*})^{2} \leq 2 \int_{\Omega} (u - \bar{u})^{2} + 2 \alpha^{2} \int_{\Omega} \bar{m}^2 \leq 2 C_{p} \int_{\Omega} |\nabla u|^{2} + \frac{2 \alpha^{2}}{|\Omega|} \left(\int_{\Omega} m\right)^{2},
    \end{equation*}
    and hence
    \begin{equation}\label{Lu2-c5}
        \frac{\gamma_{0}}{2 C_{p}} \int_{\Omega} (u - u_{*})^{2} \leq \gamma_{0} \int_{\Omega} |\nabla u|^{2} + \frac{\alpha^{2} \gamma_{0}}{C_{p} |\Omega|} \left(\int_{\Omega} m\right)^{2}\leq \gamma_{0} \int_{\Omega} |\nabla u|^{2} + \frac{\alpha^{2} \gamma_{0}\|m_0\|_{L^\infty}}{C_{p} }\int_{\Omega} m.
    \end{equation}
    Therefore, inserting \eqref{Lu2-c5} into \eqref{Lu2-c4}, we infer that
    \begin{equation}\label{Lu2-c6}
        \frac{d}{d t} \int_{\Omega} (u - u_{*})^{2} + \frac{\gamma_{0}}{2 C_{p}} \int_{\Omega} (u - u_{*})^{2} \leq c_2 \int_\Omega (|\nabla v|^2+|\nabla m|^2) + c_3 \int_{\Omega} m.
    \end{equation}
    To estimate the first term on the right hand of \eqref{Lu2-c6}, we first derive from the second and third equation of the system \eqref{sys-1} that
    \begin{equation}\label{Lu2-c7}
        \frac{d}{d t} \int_{\Omega} v^{2}  + \int_{\Omega} v^{2} + 2 \int_{\Omega} |\nabla v|^{2}\leq \theta_{2}^{2} \int_{\Omega} (u m)^{2} \leq c_1^2 \theta_{2}^{2} \|m_0\|_{L^\infty} \int_\Omega m,
    \end{equation}
    and
    \begin{equation}\label{Lu2-c8}
        \frac{d}{d t} \int_{\Omega} m^{2} +\int_{\Omega} m^{2} + 2 \int_{\Omega} |\nabla m|^{2} = - 2 \int_{\Omega} u m^{2} +\int_{\Omega} m^{2} \leq \|m_0\|_{L^\infty} \int_\Omega m.
    \end{equation}
    Thus, multiplying \eqref{Lu2-c7} and \eqref{Lu2-c8} by $c_2$  respectively, and adding them into \eqref{Lu2-c6}, one has 
    \begin{equation}\label{Lu2-c9}
        \begin{split}
            \frac{d}{d t} & \left(\int_{\Omega} (u-u_{*})^{2} + c_2 \int_{\Omega} v^{2} + c_2 \int_{\Omega} m^{2}\right)  + \frac{\gamma_{0}}{2 C_{p}} \int_{\Omega} (u - u_{*})^{2} + c_2 \int_{\Omega} v^{2} + c_2 \int_{\Omega} m^{2} \\
            & \leq (c_3 + c_1^2 c_2 \theta_2^{2} \|m_0\|_{L^\infty} + c_2 \|m_0\|_{L^\infty}) \int_{\Omega} m.
        \end{split} 
    \end{equation}
    Letting $$Z(t) := \int_{\Omega} (u-u_{*})^{2} + c_2\int_{\Omega} v^{2} + c_2 \int_{\Omega} m^{2},$$
    and $c_4=\min \left\{1, \frac{\gamma_{0}}{2 C_{p}}\right\}$, then from \eqref{Lu2-c9} we obtain 
    \begin{equation}\label{Lu2-c10}
        Z'(t) + c_{4} Z(t) \leq (c_3 + c_1^2 c_2 \theta_2^{2} \|m_0\|_{L^\infty} + c_2 \|m_0\|_{L^\infty}) |\Omega| \|m(\cdot,t)\|_{L^\infty}.
    \end{equation}
    Noting that $\|m(\cdot, t)\|_{L^\infty} \to 0$ as $t \to \infty$ (see Lemma \ref{Lmi}), we deduce from \eqref{Lu2-c10} that
    \begin{equation*}
        Z(t) \to 0, \quad \text{as} \quad t \to \infty,
    \end{equation*}
    which implies
    \begin{equation}\label{Lu2-c11}
        \lim\limits_{t\to\infty} (\|u(\cdot,t) - u_{*}\|_{L^2} + \|v(\cdot,t)\|_{L^2})=0.
    \end{equation}
    On the other hand, from Lemma \ref{Rg}, there exists a constant $c_5>0$ such that 
    \begin{equation}\label{Lu2-c12}
        \|u(\cdot,t)\|_{W^{1,\infty}}+\|v(\cdot,t)\|_{W^{1,\infty}}\leq c_5, \quad \text{for} \ t\geq 1.
    \end{equation}
    Then employing the Gagliardo-Nirenberg inequality in conjunction with \eqref{Lu2-c11} and \eqref{Lu2-c12}, we derive \eqref{Lu2-c1}. Hence, we complete the proof of Lemma \ref{Lu2-c}.
\end{proof}

\begin{proof}[\textbf{Proof of Theorem \ref{GBS-1}:  Stabilization}] 
    From Lemma \ref{Lmi} and Lemma \ref{Lu2-c}, we obtain 
    \begin{equation*}
        \lim\limits_{t\to\infty}(\|u(\cdot,t) - u_{*}\|_{L^\infty} + \|v(\cdot,t)\|_{L^\infty} + \|m(\cdot,t)\|_{L^\infty})=0,
    \end{equation*}
    which proves the global asymptotic stability of the constant steady state $(u_*,0,0)$.
\end{proof}
\section{Proof of Theorem \ref{GBS-2}: Boundedness and stabilization with $\theta_1>0$}
In this section, we shall prove the  global boundedness and global stabilization for the system \eqref{sys-1} with $\theta_1>0$ as stated in Theorem \ref{GBS-2}.
To this end, we first derive the upper bound for $v$ in the case of $\theta_1>0$ based on the comparison method developed in \cite{JLZ-2022-CPDE,FJ-2020-JDE}. 

\subsection{Auxiliary functions and properties}
To obtain the lower and upper bound for $v$, we first introduce two auxiliary functions $w$ and $S$, and then present some properties of the auxiliary functions as well as the connection between $v$ and the auxiliary functions $w$ and $S$.

Here the auxiliary function $w(x, t)$ is the unique non-negative solution of the following Helmholtz equation
\begin{equation}\label{Ew-1}
    \begin{cases}
        - \Delta w + w = u, &x \in \Omega, \ t > 0, \\
        \nabla w \cdot \nu= 0, & x \in \partial\Omega, \ t > 0.
    \end{cases}
\end{equation}
Similarly, $S(x, t)$ is the unique non-negative solution of the following system 
\begin{equation*}
    \begin{cases}
        - \Delta S + S = u + \alpha m, & x \in \Omega, \ t > 0, \\
        \nabla S \cdot\nu = 0,  &x \in \partial\Omega, \ t > 0.
    \end{cases}
\end{equation*}
For simplicity, we introduce the operator $\mathcal{A}$ on $L^{2}(\Omega)$ defined by
\begin{equation*}
    \mathcal{A} z \triangleq - \Delta z + z, \quad z \in dom (\mathcal{A})\triangleq \left\{ z \in H^{2}(\Omega) \ : \ \nabla z \cdot \nu = 0 \ \text{on} \ \partial\Omega \right\}.
\end{equation*}
Then $\mathcal{A}$ generates an analytic semigroup on $L^{p}(\Omega)$ and is invertible on $L^{p}(\Omega)$ for all $p \in (1, \infty)$. Moreover, for all $(x,t)\in \Omega \times [0, T_{\max})$, we have 
\begin{equation}\label{Dw}
    w(x,t) = \mathcal{A}^{-1} [u(x,t)] \geq 0, 
\end{equation}
and
\begin{equation}\label{DS}
    S(x,t) = \mathcal{A}^{-1} [(u + \alpha m)(x,t)] \geq 0, 
\end{equation}
where the non-negativity of $S$ and $w$ follows from the non-negativity of $u + \alpha m$ and $u$ by the comparison principle. Define
\begin{equation*}
    S_{0}(x) \triangleq S(x,0) = \mathcal{A}^{-1} [u_{0} + \alpha m_{0}](x) \quad \text{and} \quad w_{0}(x) \triangleq w(x,0) = \mathcal{A}^{-1} [u_{0}](x),
\end{equation*}
and then $S_{0}$ and $w_{0}$ both belong to $W^{3, \infty} (\Omega)$ by the regularity assumption \eqref{ID} on the initial conditions.
Building upon some ideas developed in \cite[Lemma 3.1]{JLZ-2022-CPDE}, we can derive the following key properties for $w$ and $S$.
\begin{lemma}\label{IB}
    Assume that the conditions in Lemma \ref{LS} hold and let $(u,v,m)$ be the solution of \eqref{sys-1} obtained in Lemma \ref{LS}. Then for $(x,t)\in \Omega \times (0, T_{\max})$, the following two key identities hold
    \begin{equation}\label{IB-1}
        S_{t} + u \gamma(v, m) + \alpha m = \mathcal{A}^{-1} [u \gamma(v, m) + \alpha m]
    \end{equation}
    and
    \begin{equation}\label{IB-2}
        w_{t} + u \gamma(v, m) = \mathcal{A}^{-1} [u \gamma(v, m) + \alpha u m].
    \end{equation}
    Moreover, we have
    \begin{equation}\label{IB-3}
        0<w_*\leq  w (x, t) \leq S (x, t) \leq S_{0}(x) e^{\max\left\{1, \gamma_1\right\} t}, \quad (x, t) \in \bar{\Omega} \times [0, T_{\max}),
    \end{equation}
    where $\gamma_1$ is defined in \eqref{Bg1} and $w_*>0$ is a constant independent of $\theta_1$ and $t$.
\end{lemma}
    
\begin{proof}
   First, by the definition of $\mathcal{A}$, we can rewrite the first and the third equation of \eqref{sys-1} as follows
    \begin{equation}\label{IB-4}
        u_{t} = - \mathcal{A} [u \gamma (v, m)] + u \gamma (v, m) + \alpha u m \quad \text{in} \ \Omega \times (0, T_{\max})
    \end{equation}
    and
    \begin{equation}\label{IB-5}
        m_{t} = - \mathcal{A} [m] + m - u m \quad \text{in} \ \Omega \times (0, T_{\max}).
    \end{equation}
    Then multiplying \eqref{IB-5} by $\alpha$ and adding the result into \eqref{IB-4}, we obtain
    \begin{equation}\label{IB-6}
        (u + \alpha m)_{t} + \mathcal{A} [u \gamma (v, m) + \alpha m] = u \gamma (v, m) + \alpha m.
    \end{equation}
    Therefore, the identities \eqref{IB-1} and \eqref{IB-2} follow by taking $\mathcal{A}^{-1}$ on the both sides of \eqref{IB-6} and \eqref{IB-4}, respectively.
    Using the fact $0 < \gamma(v, m) \leq \gamma_1$ in \eqref{Bg1} and the elliptic comparison principle, we infer from \eqref{IB-1} that
    \begin{equation*}
        \begin{split}
            S_{t} \leq S_{t} + u \gamma(v, m) + \alpha m = \mathcal{A}^{-1} [u \gamma(v, m) + \alpha m] & \leq \mathcal{A}^{-1} [\gamma_1 u + \alpha m] \\ 
            & \leq \max{\left\{1, \gamma_1\right\}} \mathcal{A}^{-1} [u + \alpha m] \\
            & = \max{\left\{1, \gamma_1\right\}} S,
        \end{split}
    \end{equation*}
    which yields 
    \begin{equation}\label{IB-1*}
        S (x, t) \leq S_{0}(x) e^{\max\left\{1, \gamma_1\right\} t}, \quad (x, t) \in \bar{\Omega} \times [0, T_{\max}).
    \end{equation}
    Since $u \leq u + \alpha m\leq u+\alpha m^*$ in $\Omega \times (0, T_{\max})$ by \eqref{L1-2}, it follows from \eqref{Dw}, \eqref{DS} and the elliptic comparison principle that
    \begin{equation}\label{BSw}
        w(x,t)\leq S(x,t)\leq w(x,t)+\alpha m^*, \quad (x,t)\in\bar{\Omega} \times [0, T_{\max}).
    \end{equation}
    Noting $0 < \|u_{0}\|_{L^{1}} \leq \|u(\cdot,t)\|_{L^{1}}$ in \eqref{L1-1} and applying Lemma \ref{LB}, we derive from \eqref{Ew-1} that $0< w_*\leq w$, which in conjunction with \eqref{BSw} and \eqref{IB-1*} gives \eqref{IB-3}. Hence, we complete the proof of Lemma \ref{IB}.
\end{proof}

Now we proceed to show that the upper bound of $v$ can be controlled by the auxiliary functions $w$ and $S$. Before that, we first introduce a function
\begin{equation}\label{DG}
    \Gamma^{*}(s) = \int_{1}^{s} \gamma(\eta, m^{*}) d \eta = \int_{1}^{s} \gamma(\eta, \|m_{0}\|_{L^{\infty}}) d \eta, \quad s \geq 0,
\end{equation}
and prove the following result as in \cite[Lemma 2.4]{JLZ-2022-CPDE}. 

\begin{lemma}\label{LGB1}
    Assume that $\gamma$ satisfies (H1). Then for any $\epsilon > 0$, there exists a constant $C_\epsilon> 0$ such that
    \begin{equation}\label{LGB1-1}
        s \gamma (s, m^{*}) - \gamma (s_{0}, m^{*}) \leq \Gamma^{*}(s) \leq \epsilon s + C_\epsilon, \quad s \geq s_{0} \geq 0.
    \end{equation}
\end{lemma}

\begin{proof}
    Using the conditions $\gamma(s,m)>0$ and $\gamma_s(s,m)\leq 0$ for any $s \geq s_{0} \geq 0$ in (H1), we obtain
    \begin{equation}\label{LGB1-2}
        \Gamma^{*} (s) = \int_{1}^{s} \gamma(\eta, m^{*}) d \eta \geq (s-1) \gamma (s, m^{*}) \geq s \gamma(s, m^{*}) - \gamma(s_{0}, m^{*}).
    \end{equation}
    Next, we shall prove $\Gamma^{*}(s) \leq \epsilon s + C_\epsilon$ for any $\epsilon>0$. To this end, we divide our proof into two cases: $0\leq s<1$ and $s\geq 1$.

    If $0\leq s<1$, from the positivity of $\gamma(s,m)$,
    we can find a positive constant $C_\epsilon$ such that
    \begin{equation}\label{LGB1-3}
        \Gamma^{*}(s) = \int_{1}^{s} \gamma(\eta, m^{*}) d \eta< 0 \leq  \epsilon s + C_\epsilon.
    \end{equation}
    On the other hand, we consider the case $s\geq 1$. From (H1), we have $\lim\limits_{s\to + \infty}\gamma(s,m)=0$. Then for any $\epsilon > 0$ and $m\geq 0$, there exists a constant $s_{\epsilon} > 1$ such that $\gamma (s, m^{*}) \leq \epsilon$ for $s \geq s_{\epsilon}$. Thus, since $\gamma_s(s,m)\leq 0$, we derive that
    \begin{equation}\label{LGB1-4}
        \Gamma^{*} (s) = \int_{1}^{s_{\epsilon}} \gamma(\eta, m^{*}) d \eta + \int_{s_{\epsilon}}^{s} \gamma(\eta, m^{*}) d \eta \leq s_{\epsilon} \gamma (1, m^{*}) + \epsilon (s - s_{\epsilon})_{+} \leq \epsilon s+s_{\epsilon} \gamma (1, m^{*}).
    \end{equation}
    Consequently, the combination of \eqref{LGB1-2}, \eqref{LGB1-3} and \eqref{LGB1-4} gives \eqref{LGB1-1}. Hence, we complete the proof of Lemma \ref{LGB1}.
\end{proof}

Subsequently, in the case $\theta_{1} > 0$, we establish an upper bound for $v$ in terms of $S$ through an application of the comparison principle. More precisely, we obtain the following result:
\begin{lemma}\label{LvS1}
    Let $(u,v,m)$ be the solution of \eqref{sys-1} obtained in Lemma 2.1 and $\theta_{1} > 0$, $\theta_{2} \geq 0$. 
    Assume that $\gamma$ satisfies the assumption (H1).
    Then for any $\varepsilon>0$, there exists a constant $K_\varepsilon>0$ independent of $\theta_1$ and $t$ such that
    \begin{equation}\label{LvS1-1}
        v \leq (1 + \varepsilon) (\theta_1+\theta_2 m^*) S + (1 + \theta_1) K_\varepsilon, \quad \text{in} \ \bar{\Omega} \times [0, T_{\max}).
    \end{equation}
    Moreover, we can find a positive constant $\mathcal{K}_1$ independent of $\theta_1$ and $t$ such that 
    \begin{equation}\label{LvSB}
        v\leq (1+\theta_1)\mathcal{K}_1 w\leq (1+\theta_1)\mathcal{K}_1 S,\quad \text{in} \ \bar{\Omega} \times [0, T_{\max}).
    \end{equation}
\end{lemma}

\begin{proof}
    Introducing the parabolic operator
    \begin{equation*}
        \mathcal{L} z \triangleq \partial_{t} z + \mathcal{A} z = \partial_{t} z - \Delta z + z.
    \end{equation*}
    Then from the second equation of \eqref{sys-1} and $-\Delta S+S=u + \alpha m $, we can derive that
    \begin{equation*}
        \mathcal{L} v = (\theta_{1} + \theta_{2} m) u \leq (\theta_{1} + \theta_{2}m) (u + \alpha m) = (\theta_{1} + \theta_{2}m) (\mathcal{L} S - \partial_{t} S),
    \end{equation*}
    which, combined with \eqref{L1-2}, \eqref{IB-1} and the fact $\mathcal{A}^{-1} [u \gamma(v, m) + \alpha m]\geq 0$ by the elliptic comparison principle, gives
    \begin{equation}\label{LvS1-2}
        \begin{split}
            \mathcal{L} v & \leq (\theta_{1} + \theta_{2}m) \mathcal{L} S -(\theta_{1} + \theta_{2} m)\partial_{t} S\\
            &=  (\theta_{1} + \theta_{2}m) \mathcal{L} S + (\theta_{1} + \theta_{2}m) (u \gamma(v, m) + \alpha m) - (\theta_{1} + \theta_{2}m) \mathcal{A}^{-1} [u \gamma(v, m) + \alpha m]\\
            &\leq (\theta_{1} + \theta_{2} m) \mathcal{L} S + (\theta_{1} + \theta_{2} m) (u \gamma(v, m) + \alpha m)\\
            &\leq  (\theta_{1} + \theta_{2} m) \mathcal{L} S +\gamma(v, m^*)(\theta_{1} + \theta_{2} m)u+\alpha(\theta_{1} + \theta_{2} m^*) m^{*}.
        \end{split}
    \end{equation}
    For any $(x,t) \in \Omega \times (0,T_{\max})$, letting
    \begin{equation*}
        K = \left\{
        \begin{aligned}
            & \theta_{1}, \quad && \text{if} \ \mathcal{L} S \leq 0, \\
            & \theta_{1} + \theta_{2} m^{*}, \quad && \text{if} \ \mathcal{L} S \geq 0,
        \end{aligned}
        \right.
    \end{equation*}
    we deduce that
    \begin{equation}\label{LvS1-3}
        (\theta_{1} + \theta_{2}m) \mathcal{L} S \leq K \mathcal{L} S.
    \end{equation}
    Moreover, using the second equation of \eqref{sys-1}, the assumption (H1) on $\gamma$ and the definition of $\Gamma^{*}$ in \eqref{DG}, we obtain
    \begin{equation*}
        \gamma(v, m^*)(\theta_{1} + \theta_{2} m)u = \gamma(v, m^{*}) (\partial_{t} v - \Delta v + v) = \mathcal{L} \Gamma^{*}(v) + \gamma_{v}(v, m^{*}) | \nabla v |^{2} + (v \gamma(v, m^{*}) - \Gamma^{*} (v)),
    \end{equation*}
    which, together with Lemma \ref{LGB1} (with $s_{0} = 0$) and $\gamma_{v}(v, m^{*})\leq 0$ by (H1), yields
    \begin{equation}\label{LvS1-4}
        \gamma(v, m^*)(\theta_{1} + \theta_{2} m)u \leq \mathcal{L} \Gamma^{*}(v) + \gamma(0,m^*).
    \end{equation}
    Then substituting \eqref{LvS1-3} and \eqref{LvS1-4} into \eqref{LvS1-2}, one has
    \begin{equation}\label{LvS1-5}
        \mathcal{L} v \leq \mathcal{L} [K S + \Gamma^{*}(v)] + \gamma(0, m^{*}) + \alpha m^* (\theta_1 + \theta_2 m^*) = \mathcal{L} [K S + \Gamma^{*}(v)+M],
    \end{equation}
    where $$M= \gamma(0, m^{*}) + \alpha m^* (\theta_1 + \theta_2 m^*).$$
    Therefore, letting $M_1=\|v_0\|_{L^\infty}+M$, we infer from \eqref{LvS1-5} and the boundary condition that 
    \begin{equation}\label{LvS1-6}
        \mathcal{L}v\leq \mathcal{L} [K S + \Gamma^{*}(v)+M]\leq \mathcal{L}[K S + \Gamma^{*}(v)+M_1], \quad \text{in} \ \Omega \times (0,T_{\max}),
    \end{equation}
    and
    \begin{equation}
        \nabla v \cdot \nu = \nabla(K S + \Gamma^{*}(v)+M_1) \cdot \nu = 0, \quad \text{on} \ \partial \Omega \times (0,T_{\max}),
    \end{equation}
    as well as
    \begin{equation}\label{LvS1-8}
        v_0\leq \|v_0\|_{L^\infty}\leq K S_0 + \Gamma^{*}(v_0)+M_1, \quad \text{in} \ \bar{\Omega}.
    \end{equation}
    Applying the standard parabolic comparison principle to \eqref{LvS1-6}-\eqref{LvS1-8}, it follows that
    \begin{equation*}
        v(x, t) \leq K S(x, t) + \Gamma^{*}(v(x, t)) +M_1, \quad (x, t) \in \bar{\Omega} \times [0, T_{\max}),
    \end{equation*}
    and hence
    \begin{equation}\label{LvS1-9}
        v(x, t) \leq (\theta_1+\theta_2 m^*) S(x, t) + \Gamma^{*}(v(x, t)) +(1+\theta_1) M_2, \quad (x, t) \in \bar{\Omega} \times [0, T_{\max}),
    \end{equation}
    where $M_2=\|v_0\|_{L^\infty} + \gamma(0, m^{*}) + \alpha m^* (1 + \theta_2 m^*)$.
    From Lemma \ref{LGB1}, we know that for any $\epsilon\in(0,1)$, it holds that 
    \begin{equation}\label{LvS1-10}
        \Gamma^{*}(v(x,t)) \leq \epsilon v(x,t) + C_\epsilon, \quad (x,t) \in \bar{\Omega} \times [0,T_{\max}).
    \end{equation}
    Thus, substituting \eqref{LvS1-10} into \eqref{LvS1-9}, we obtain
    \begin{equation*}
        v(x, t) \leq \frac{1}{1-\epsilon}(\theta_1+\theta_2 m^*) S(x, t) +\frac{(1+\theta_1)}{1-\epsilon}(M_2+C_\epsilon),
    \end{equation*}
    which implies \eqref{LvS1-1} by letting $\varepsilon=\frac{\epsilon}{1-\epsilon}$ and
    \begin{equation*}
        K_\varepsilon=(1+\varepsilon)(M_2+C_\epsilon).
    \end{equation*}
    Next, we shall prove \eqref{LvSB}. In fact, letting $\varepsilon=1$ in \eqref{LvS1-1} and noting the fact $0<w_*\leq w\leq S\leq w+\alpha m^*$ from \eqref{IB-3} and \eqref{BSw}, we show that 
    \begin{equation*}
        \begin{split}
            v &\leq 2(\theta_1+\theta_2 m^*) S + (1 + \theta_1) K_1 \\
            &\leq 2(\theta_1+\theta_2 m^*) (w+\alpha m^*)+(1 + \theta_1) K_1\\
            & \leq (1+\theta_1)\left[2(1+\theta_2m^*)+\frac{2(1+\theta_2m^*)\alpha m^* +K_1}{w_*}\right]w \\
            &\leq \frac{2(1+\theta_2 m^*)(w_*+\alpha m^*)+K_1}{w_*}(1+\theta_1)w,
        \end{split}
    \end{equation*}
    which gives \eqref{LvSB} by $w\leq S$ in \eqref{BSw} and letting 
    \begin{equation*}
        \mathcal{K}_1:=\frac{2(1+\theta_2 m^*)(w_*+\alpha m^*)+K_1}{w_*}.
    \end{equation*}
\end{proof}   
\subsection{Global boundedness: $\theta_1>0$} In this subsection, we shall prove the global boundedness of classical solutions for the system \eqref{sys-1} with $\theta_1 > 0$. To achieve this, the key point is to find a positive upper bound for $v$. To this end, we first show that if $\gamma(v, m)$ decays exponentially for all $m \geq 0$ as $v \to \infty$ and if the mass of the initial data is less than a  value, then there exists a constant $C > 0$ independent of $t$ such that $\|v(\cdot, t)\|_{L^\infty} \leq C$. However, by Lemma \ref{LvS1}, the boundedness of $\|v(\cdot, t)\|_{L^\infty}$ can be derived if $\|w(\cdot, t)\|_{L^\infty} \leq C$, where $w(\cdot, t) = \mathcal{A}^{-1}[u(\cdot, t)]$ satisfies  
\begin{equation}\label{WE}
    \begin{cases}
        \partial_{t} w + u \gamma(v, m)  = \mathcal{A}^{-1} [u \gamma(v, m) + \alpha u m], \quad & x \in \Omega, \ t >0, \\ 
        - \Delta w + w = u, \quad & x \in \Omega, \ t >0, \\ 
        \nabla w \cdot \nu = 0, \quad & x \in \partial \Omega, \ t >0,\\
	w(x,0) = w_{0}(x), \quad & x \in \Omega.
    \end{cases}
\end{equation}

\begin{lemma}\label{L2B1}
    Let $(u, v, m)$ be the solution of system \eqref{sys-1} with $\theta_1>0$ obtained in Lemma \ref{LS}. Assume that $\gamma$ satisfies (H1) and there is $\chi>0$ such that
    \begin{equation}\label{L2B1-1}
        \liminf_{s \to + \infty} e^{\chi s} \gamma (s, m) > 0, \quad \text{for all } m \geq 0.
    \end{equation}
    And suppose that the initial data satisfies \eqref{ID} and 
    \begin{equation}\label{L2B1-2}
        \|u_{0} + \alpha m_{0}\|_{L^{1}} < \frac{4\pi}{\chi (\theta_1 + \theta_2 m^*)}.
    \end{equation}
    Then there exist two positive constants $C_{1}$ and $C_2$, independent of $\theta_1$ and $t$, such that
     \begin{equation}\label{L2B1-3}
     \int_{t}^{t + \tau} \int_{\Omega} u^{2} \gamma(v, m) d x d s\leq C_{1} e^{C_2 \theta_1}, \quad \text{for} \ t\in (0,\widetilde{T}_{\max}),
    \end{equation}
    where $\tau$ and $\widetilde{T}_{\max}$ are defined in \eqref{DT}.
\end{lemma}

\begin{proof} 
    We shall prove this lemma based on some ideas in  \cite{FJ-2021-CVPDE}. In fact, using the facts that $w(\cdot,t) = \mathcal{A}^{- 1} [u(\cdot,t)]$ in \eqref{Dw} and $\mathcal{A}$ is a self-adjoint operator, we multiply the first equation of \eqref{WE} by $u$ and integrate it by parts such that
    \begin{equation}\label{L2B1-5}
        \begin{split}
            \frac{1}{2} \frac{d}{d t} (\|\nabla w\|_{L^{2}}^{2} + \|w\|_{L^{2}}^{2}) + \int_{\Omega} \gamma(v, m) u^{2} &= \int_{\Omega} \mathcal{A}^{- 1}[\alpha u m + u \gamma(v, m)] u \\
            &= \int_{\Omega} (\alpha u m + u \gamma(v, m)) w\\
            &\leq (\alpha m^* + \gamma_1)\int_{\Omega} u w,
        \end{split}
    \end{equation}
    by combining $\|m(\cdot,t)\|_{L^\infty}\leq m^*$ in \eqref{L1-2} with  $0<\gamma(v,m)\leq \gamma_1$ in \eqref{Bg1}.
    On the other hand, multiplying $-\Delta w +w=u$ by $w$ and integrating the result by parts, we derive
    \begin{equation}\label{L2B1-6}
        \|\nabla w\|_{L^{2}}^{2} + \|w\|_{L^{2}}^{2} = \int_{\Omega} u w.
    \end{equation}
    Therefore, adding \eqref{L2B1-5} into \eqref{L2B1-6}, and applying Young's inequality, we have
    \begin{equation*}
        \begin{split}
            & \frac{1}{2} \frac{d}{d t} (\|\nabla w\|_{L^{2}}^{2} + \|w\|_{L^{2}}^{2}) + (\|\nabla w\|_{L^{2}}^{2} + \|w\|_{L^{2}}^{2}) + \int_{\Omega} \gamma(v, m) u^{2} \\
            & \leq (\alpha m^* + \gamma_1 + 1) \int_{\Omega} u w \\
            & \leq \frac{1}{2} \int_{\Omega} \gamma(v, m) u^{2} + \frac{(\alpha m^* + \gamma_1 + 1)^2}{2} \int_{\Omega} \frac{w^2}{\gamma(v, m)},
        \end{split}
    \end{equation*}
    which, in conjunction with $\gamma(v,m)\geq \gamma(v,0)>0$ by the non-negativity of $m$ and the assumption (H1), gives
    \begin{equation}\label{L2B1-7}
        \frac{d}{d t} (\|\nabla w\|_{L^{2}}^{2} + \|w\|_{L^{2}}^{2}) + (\|\nabla w\|_{L^{2}}^{2} + \|w\|_{L^{2}}^{2}) + \int_{\Omega} \gamma(v, m) u^{2} \leq c_{1} \int_{\Omega} \frac{w^2}{\gamma(v, 0)},
    \end{equation}
    where $c_{1} = (\alpha m^* + \gamma_1 + 1)^2$. 
    Noting that $\gamma$ satisfies the assumptions \eqref{L2B1-1} and (H1), then we fix $\eta \geq 0$, there exist $b > 0$ and $s_b>0$ 
    depending on $b$ and $\chi$ such that
    \begin{equation}\label{L2B1-8}
        \gamma^{- 1}(s, \eta) \leq b e^{\chi s} + \gamma^{-1} (s_{b}, \eta), \quad \text{for all} \ s\geq 0. 
    \end{equation}
    From \eqref{L2B1-8} and $v \leq (1 + \varepsilon) (\theta_1+\theta_2 m^*) S + (1 + \theta_1) K_\varepsilon$ for any $\varepsilon>0$ in \eqref{LvS1-1}, we deduce that 
    \begin{equation}\label{L2B1-9}
        \begin{split}
            \int_{\Omega} \frac{w^2}{\gamma(v, 0)} &\leq \int_{\Omega} [b e^{\chi v} + \gamma^{-1} (s_b, 0)] w^{2} \\
            &\leq \int_{\Omega} [b e^{\chi (1+\varepsilon) (\theta_1+\theta_2m^*) S + \chi (1+\theta_1)K_\varepsilon} + \gamma^{-1} (s_b, 0)] w^{2}\\
            &\leq be^{\chi (1+\theta_1)K_\varepsilon}\left(\int_{\Omega} e^{p \chi (1+\varepsilon) (\theta_1+\theta_2m^*) S}\right)^{\frac{1}{p}} \left(\int_\Omega w^{2p'}\right)^{\frac{1}{p'}} + \gamma^{-1} (s_b, 0) \int_{\Omega} w^{2},
        \end{split}
    \end{equation}
    where $p>1$ is close to 1 satisfying $\frac{1}{p}+\frac{1}{p'}=1$. Combining $\|u(\cdot,t)\|_{L^1}\leq \|u_0+\alpha m_0\|_{L^1}$ in \eqref{L1-1} with the fact that $w$ satisfies \eqref{Ew-1}, we invoke \eqref{LB-3} in Lemma \ref{LB} such that  
    \begin{equation}\label{L2B1-9*}
        \left(\int_\Omega w^{2p'}\right)^{\frac{1}{p'}}\leq c_2 \quad \text{and} \quad \int_\Omega w^2\leq c_2.
    \end{equation}
    On the other hand, noting that $\|S(\cdot,t)\|_{L^1}=\|u_0+\alpha m_0\|_{L^1}:=\Lambda$ and \eqref{L2B1-2}, we can find some sufficient small $\varepsilon>0$ and fix $p>1$ close to 1 such that
    \begin{equation*}
        (1+\varepsilon) p \Lambda<\frac{4\pi}{\chi(\theta_1+\theta_2 m^*)}
    \end{equation*}
    and hence 
    \begin{equation}\label{L2B1-10}
        \chi (1+\varepsilon) (\theta_1+\theta_2m^*) p < \frac{4 \pi}{\Lambda}.
    \end{equation}
    Since \eqref{L2B1-10} holds, we infer from \eqref{LB-4} in Lemma \ref{LB} that
    \begin{equation}\label{L2B1-11}
        \left(\int_{\Omega} e^{p \chi (1+\varepsilon) (\theta_1+\theta_2m^*) S}\right)^\frac{1}{p}\leq c_3. 
    \end{equation}
    Therefore, inserting \eqref{L2B1-9*} and \eqref{L2B1-11} into \eqref{L2B1-9}, one has
    \begin{equation*}
        \int_{\Omega} \frac{w^2}{\gamma(v, 0)} \leq c_{4} e^{c_5 \theta_1},
    \end{equation*}
    which, substituted it into \eqref{L2B1-7}, gives a constant $c_{6}> 0$ independent of $t$ and $\theta_1$ such that
    \begin{equation}\label{L2B1-12}
        \frac{d}{d t} (\|\nabla w\|_{L^{2}}^{2} + \|w\|_{L^{2}}^{2}) + (\|\nabla w\|_{L^{2}}^{2} + \|w\|_{L^{2}}^{2}) + \int_{\Omega} \gamma(v, m) u^{2} \leq c_{6} e^{c_5\theta_1}.
    \end{equation}
    Letting $Z(t) = \|\nabla w\|_{L^{2}}^{2} + \|w\|_{L^{2}}^{2}$ and $h(t) = c_{6}e^{c_5 \theta_1}$ and employing \eqref{L2B1-12}, we obtain
    \begin{equation*}
        Z'(t) + Z(t) \leq h(t) \quad \text{for all} \ t \in (0, T_{\max}),
    \end{equation*}
    which implies
    \begin{equation}\label{L2B1-13}
        \|\nabla w(\cdot,t)\|_{L^{2}} + \|w(\cdot,t)\|_{L^{2}} \leq c_{7}e^{c_8\theta_1},
    \end{equation}
    where $c_{7}$ and $c_8$ are positive constants independent of $t$ and $\theta_1$.
    
    Moreover, integrating \eqref{L2B1-12} over the interval $(t,t+\tau)$ and using \eqref{L2B1-13}, we derive \eqref{L2B1-3} directly. Then we complete the proof of Lemma \ref{L2B1}.
\end{proof}

\begin{remark}
    If $\gamma(v,m)$ satisfies \eqref{L2B1-1} for any $\chi>0$, then we can directly obtain \eqref{L2B1-11} and hence \eqref{L2B1-3} holds without any smallness assumption on the initial data.
\end{remark}

\begin{lemma}\label{L2B2}
    Under the assumptions in Lemma \ref{L2B1}, it follows that
    \begin{equation}\label{L2B2-1}
         \|v(\cdot, t)\|_{L^{\infty}}\leq C_3 (1+\theta_1) e^{C_4\theta_1}, \quad \text{for} \ t\in (0,T_{\max}),
    \end{equation}
    where $C_3$ and $C_4$ are positive constants independent of $\theta_1$ and $t$.  
\end{lemma}

\begin{proof}
Using the continuous embedding of $W^{2, p}(\Omega) \hookrightarrow L^{\infty}(\Omega)$ ($p>1$) and the standard elliptic regularity, we derive from \eqref{Ew-1} that
\begin{equation}\label{wL}
    \|w(\cdot,t)\|_{L^\infty} \leq c_1 \|w(\cdot,t)\|_{W^{2,p}} \leq c_2 \|u(\cdot,t)\|_{L^{p}},
\end{equation}
with some $1<p<2$ such that $\frac{p}{2-p}$ is larger than 1 but near 1. Moreover, we can use the H{\"o}lder inequality to derive 
\begin{equation*}
\|u(\cdot,t)\|_{L^{p}}\leq \|u \gamma^{\frac{1}{2}}(v, m)\|_{L^{2}} \|\gamma^{-\frac{1}{2}}(v, m)\|_{L^{\frac{2 p}{2-p}}},
\end{equation*}
which substituted into \eqref{wL} gives
\begin{equation}\label{L2B2-2}
\|w(\cdot,t)\|_{L^\infty} \leq c_2 \|u \gamma^{\frac{1}{2}}(v, m)\|_{L^{2}} \|\gamma^{-\frac{1}{2}}(v, m)\|_{L^{\frac{2 p}{2-p}}}.
\end{equation}
Under the facts  \eqref{L2B1-8} and $v \leq (1 + \varepsilon) (\theta_1+\theta_2 m^*) S + (1 + \theta_1) K_\varepsilon$ for any $\varepsilon>0$ in \eqref{LvS1-1}, we obtain
\begin{equation}\label{L2B2-2-1}
\|\gamma^{-\frac{1}{2}}(v, m)\|_{L^{\frac{2 p}{2-p}}} \leq c_3\left[\int_\Omega \left(b e^{\chi(1+\theta_1)K_\varepsilon}\right)^\frac{p}{2-p} e^{\chi (1+\varepsilon) (\theta_1+\theta_2m^*) S \frac{p}{2-p}} +(\gamma^{-1}(s_b,0))^{\frac{p}{2-p}}\right]^{\frac{2-p}{2 p}}
\end{equation}
Noting $\|S(\cdot, t)\|_{L^1} = \|(u_0+\alpha m_0)\|_{L^1}:= \Lambda$ and using the condition \eqref{L2B1-2} that $$\Lambda=\Vert u_{0} + \alpha m_{0} \Vert_{L^{1}} < \frac{4\pi}{\chi (\theta_1 + \theta_2m^*)},$$
we can choose some $\varepsilon>0$ sufficient small such that
\begin{equation*}
    (1+\varepsilon) \frac{p}{2-p} \Lambda< \frac{4 \pi}{\chi (\theta_1+\theta_2m^*)}
\end{equation*}
and hence
\begin{equation}\label{L2B2-2*}
    \chi (1+\varepsilon) (\theta_1+\theta_2m^*)\frac{p}{2-p} <\frac{4\pi}{\Lambda}.
\end{equation}
Utilizing \eqref{L2B2-2*}, we infer from \eqref{LB-4} in Lemma \ref{LB} that
\begin{equation*}
\int_\Omega  e^{\chi (1+\varepsilon) (\theta_1+\theta_2m^*) S \frac{p}{2-p}} \leq c_6,
\end{equation*}
which substituted into \eqref{L2B2-2-1} gives
\begin{equation}\label{L2B2-2-2}
\|\gamma^{-\frac{1}{2}}(v, m)\|_{L^{\frac{2 p}{2-p}}}\leq c_7(e^{c_8\theta_1}+1)\leq c_9 e^{c_8\theta_1}.
\end{equation}
Combining \eqref{L2B2-2} and \eqref{L2B2-2-2}, and using Young's inequality, we have
\begin{equation}\label{L2B2-3}
\|w(\cdot,t)\|_{L^\infty} \leq c_9 e^{c_8\theta_1} \left(\int_{\Omega} u^2 \gamma(v, m)\right)^{\frac{1}{2}} \leq \int_{\Omega} u^2 \gamma(v, m)+ c_{10}e^{c_{11} \theta_1}
\end{equation}
With \eqref{L2B1-3} in hand, we can derive from \eqref{L2B2-3} that
    \begin{equation*}
    \int_{t}^{t+\tau} \|w(\cdot,s)\|_{L^\infty} \leq \int_{t}^{t+\tau} \int_{\Omega} u^2 \gamma(v, m) d x d s  + c_{10}e^{c_{11} \theta_1} \leq c_{12} e^{c_{13}\theta_1}, \quad \text{for all} \ t\in (0,T_{\max} - \tau),
    \end{equation*}
    where $\tau$ is defined in \eqref{DT} and hence for any fixed $x \in \Omega$,
    \begin{equation}\label{L2B2-3*}
    \int_{t}^{t+\tau} w(x,s) d s \leq \int_{t}^{t+\tau} \|w(\cdot,s)\|_{L^\infty} d s \leq c_{12} e^{c_{13}\theta_1}.
    \end{equation}
    Moreover, in view of $0\leq m\leq m^*$ in \eqref{L1-2} and $0<\gamma(v,m)\leq \gamma_1$ in \eqref{Bg1}, applying the elliptic comparison principle to \eqref{IB-2}, we obtain
    \begin{equation*}
        w_{t} + u \gamma(v, m) = \mathcal{A}^{-1} [u \gamma(v, m) + \alpha u m] \leq (\gamma_1 + \alpha m^*) w,
    \end{equation*}
    which, in conjunction with the non-negativity of $u$, implies
    \begin{equation}\label{L2B2-3a}
        w_{t}\leq (\gamma_1 + \alpha m^*) w.
    \end{equation}
    For any $t\in (0,T_{\max})$, we deduce from \eqref{L2B2-3*} that there exists $t_0\in ((t-\tau)_+,t)$ satisfying $t_0\geq 0$ and $t_0\in (0,T_{\max} - \tau)$, such that
    \begin{equation}\label{L2B2-3b}
        w(\cdot,t_0)\leq c_{14} e^{c_{13}\theta_1}.
    \end{equation}
    Then integrating \eqref{L2B2-3a} over $(t_0,t)$ and employing \eqref{L2B2-3b} and $t\leq t_0+\tau\leq t_0+1$, we obtain 
    \begin{equation}\label{L2B2-3c}
        w(\cdot,t) \leq w(\cdot,t_0) e^{\int_{t_0}^t (\gamma_1 + \alpha m^*) ds} \leq c_{15} e^{c_{13}\theta_1}.
    \end{equation}
    Owing to \eqref{IB-3} for $t \in [0,\tau]$ and \eqref{L2B2-3c}, we infer that
    \begin{equation}\label{L2B2-4}
        w(x, t) \leq c_{16} e^{c_{13}\theta_1}, \quad \text{for} \ t\in[0,T_{\max}).
    \end{equation}
    Hence, it follows from \eqref{L2B2-4} and $v\leq (1+\theta_1)\mathcal{K}_1 w$ in \eqref{LvSB} that
    \begin{equation*}
        v(x, t) \leq (1+\theta_1) \mathcal{K}_1 w(x,t) \leq c_{17}(1+\theta_1) e^{c_{13}\theta_1}, \quad \text{for} \ t \in[0,T_{\max}),
    \end{equation*}
    which yields \eqref{L2B2-1}.
\end{proof}

\begin{remark}
    For any bounded positive constant $\theta_1$, we can assume without loss of generality that $0 < \theta_1 \leq c_1$. Consequently, from Lemma \ref{L2B2}, we know that there exists a constant $c_2 > 0$ independent of $\theta_1$ such that 
    \begin{equation}\label{Lu2-a1-1}
        0 < v \leq c_2.
    \end{equation}
    For clarity, we denote $\gamma_0$ as the lower bound of $\gamma(v,m)$. Under the assumption (H1) on $\gamma(v,m)$ and using \eqref{Lu2-a1-1}, there exist constants $\gamma_i (i=0,1,2,3)$ independent of $t$ and $\theta_1$ such that 
    \begin{equation}\label{Lu2-a1-2}
        0 < \gamma_0 \leq \gamma(v,m)\leq \gamma_1, \quad |\gamma_v(v,m)| \leq \gamma_2, \quad \text{and} \quad |\gamma_m(v,m)| \leq \gamma_3.
    \end{equation}
\end{remark}
Next, we shall use arguments similar to those for $\theta_1=0$ to obtain the boundedness of $\|u(\cdot,t)\|_{L^\infty}$. For brevity, we only present the outline and omit the details of the proof. More precisely, we establish the following results.
\begin{lemma}\label{LIa}
    Suppose that the conditions in Lemma \ref{L2B1} hold. Let $\theta_1>0$ and $(u, v, m)$ be the solution of the system \eqref{sys-1} obtained in Lemma \ref{LS}. Then it holds that
    \begin{equation*}
        \|u(\cdot, t)\|_{L^{\infty}} \leq M:=C_{5} (1+\theta_1)^{12} e^{C_{6} (1+\theta_1)^6} \quad \text{for all} \ t \in (0, T_{\max}),
    \end{equation*}
    where the constants $C_{5},C_{6} > 0$ are independent of $t$ and $\theta_1$.
\end{lemma}

\begin{proof}
    By virtue of \eqref{Lu2-a1-2}, we can adopt arguments similar to those used in proving Lemma \ref{Lut} to show that there exists a constant $c_1 > 0$ independent of $t$ such that
    \begin{equation}\label{LIa-1}
        \int_{t}^{t + \tau} \int_{\Omega} u^{2} \leq c_1,
    \end{equation}
    where
    \begin{equation*}
        \tau := \min \left\{1, \frac{1}{2} T_{\max}\right\} \quad \text{and} \quad \widetilde{T}_{\max} := T_{\max} - \tau.
    \end{equation*}
    Analogously to the proof of Lemma \ref{Lvm}, we derive the following inequalities:
    \begin{equation}\label{LIa-2}
        \frac{d}{dt} \int_{\Omega} |\nabla v|^{2} + 2 \int_{\Omega} |\nabla v|^{2} + \int_{\Omega} |\Delta v|^{2} \leq \int_{\Omega} [(\theta_1 + \theta_2 m)u]^{2} \leq (1 + \theta_1)^2 (1 + \theta_2 m^*)^2 \int_\Omega u^2,
    \end{equation}
    and
    \begin{equation}\label{LIa-3}
        \frac{d}{dt}\int_\Omega |\nabla m|^2 + 2\int_\Omega |\nabla m|^2 + \int_\Omega |\Delta m|^2 \leq 2\|m_0\|^2_{L^\infty}\int_\Omega u^2 + 2\|m_0\|^2_{L^\infty}|\Omega|.
    \end{equation}
    Invoking \eqref{LIa-1} and applying Lemma \ref{LB3} to \eqref{LIa-2} and \eqref{LIa-3}, we find two positive constants $c_2$ and $c_3$ independent of $\theta_1$ and $t$ such that
    \begin{equation}\label{LIa-4}
        \int_{\Omega} |\nabla v|^{2} + \int_\Omega |\nabla m|^2 \leq c_2 (1 + \theta_1)^2  \quad \text{for all} \ t \in (0, T_{\max}),
    \end{equation}
    and
    \begin{equation}\label{LIa-5}
        \int_{t}^{t + \tau} \int_{\Omega} |\Delta v|^{2} + \int_{t}^{t + \tau} \int_{\Omega} |\Delta m|^{2} \leq c_3(1 + \theta_1)^2 \quad \text{for all} \ t \in (0, \widetilde{T}_{\max}).
    \end{equation}
    Using \eqref{LIa-4} and \eqref{LIa-5}, we employ arguments similar to those in Lemma \ref{Lu2} to obtain
    \begin{equation*}
        \|u(\cdot, t)\|_{L^{2}} \leq c_4 e^{c_5 (1 + \theta_1)^6 }.
    \end{equation*}
    Finally, following the standard steps in Lemmas \ref{Lu4} and \ref{LI}, we conclude that
    \begin{equation*}
        \|u(\cdot, t)\|_{L^\infty} \leq c_6 (1 + \theta_1)^{12} e^{c_7 (1 + \theta_1)^6 }.
    \end{equation*}
    This completes the proof of Lemma \ref{LIa}.
\end{proof}

\begin{proof}[\textbf{Proof of Theorem \ref{GBS-2}: Boundedness}]
    For any $\theta_1>0$, from Lemma \ref{LIa}, we can find a constant $M>0$ independent of $t$ such that $\|u(\cdot,t)\|_{L^\infty}\leq M$, which combined with the local existence results in Lemma \ref{LS} proves the existence of global classical solution with uniform-in-time bound stated in Theorem \ref{GBS-2}.
\end{proof}

\subsection{Global stabilization: $\theta_1>0$}
In this subsection, we proceed to analyze the asymptotic behavior of solutions for the system \eqref{sys-1} with $\theta_1>0$. First, we should point out that for $\theta_1>0$, the result in Lemma \ref{Lmi} still holds and hence 
\begin{equation}\label{Cw}
    \lim\limits_{t\to\infty}\|m(\cdot,t)\|_{L^\infty}=0.
\end{equation}
Henceforth, we only focus on the convergence of $u$ and $v$. In fact, we establish the following results:
\begin{lemma}\label{Lu2-a}
    Let $(u,v,m)$ be the solution of the system \eqref{sys-1} with $\theta_1>0$ obtained in Theorem \ref{GBS-2}. Then there exists a $\theta_*>0$ such that if $0<\theta_1\leq \theta_*$ we have
\begin{equation}\label{Lu2-a1}
        \lim_{t\to \infty} (\|u(\cdot, t) - u_*\|_{L^\infty} + \|v(\cdot, t)- \theta_1 u_*\|_{L^\infty}) = 0,
    \end{equation}
    where $u_* = \frac{1}{|\Omega|} \int u_0 + \frac{\alpha}{|\Omega|} \int_{\Omega} m_0$.
\end{lemma}

\begin{proof}
    As in Lemma \ref{Lu2-c}, we first rewrite the first equation of the system \eqref{sys-1} as 
    $$(u - u_{*})_{t} = \Delta (u \gamma (v, m)) + \alpha u m,$$
    and then multiply it by $u - u_{*}$ and integrate it by parts to derive that 
    \begin{equation*}
        \frac{1}{2} \frac{d}{d t} \int_{\Omega} (u-u_{*})^{2} + \int_{\Omega}\gamma(v, m) |\nabla u|^{2} = - \int_{\Omega} \gamma_{v} u \nabla u \cdot \nabla v - \int_{\Omega} \gamma_{m} u \nabla u \cdot \nabla m + \alpha \int_{\Omega} u m (u-u_{*}),
    \end{equation*}
    which, together with the facts \eqref{Lu2-a1-2} and $\|u(\cdot,t)\|_{L^\infty}\leq M$ in Lemma \ref{LIa}, gives
    \begin{equation*}
        \begin{split}
            &\frac{d}{d t} \int_{\Omega} (u - u_{*})^{2} + 2\gamma_{0} \int_{\Omega} |\nabla u|^{2} \\
            &\leq 2 M\gamma_2\int_\Omega |\nabla u||\nabla v| + 2 M\gamma_3 \int_\Omega |\nabla u| |\nabla m| + 2 \alpha M^2 \int_{\Omega} m\\
            &\leq \gamma_0 \int_\Omega |\nabla u|^2 + \frac{2 M^2(\gamma_2^2 + \gamma_3^2)}{\gamma_0} \int_\Omega (|\nabla v|^2+|\nabla m|^2)+2 \alpha M^2 \int_{\Omega} m,
        \end{split}
    \end{equation*}
    and then 
    \begin{equation}\label{Lu2-a1-3}
        \begin{split}
            \frac{d}{d t} \int_{\Omega} (u - u_{*})^{2} + \gamma_{0} \int_{\Omega} |\nabla u|^{2} 
            &\leq \frac{2M^2(\gamma_2^2+\gamma_3^2)}{\gamma_0}\int_\Omega (|\nabla v|^2+|\nabla m|^2)+2 \alpha M^2 \int_{\Omega} m.
            \end{split}
		\end{equation}
    Recalling that $\bar{u}=\frac{1}{|\Omega|}\int_\Omega u= u_*- \frac{\alpha}{|\Omega|} \int_\Omega m=u_*-\alpha \bar{m}$, we use the Poincar$\acute{\text{e}}$ inequality with the constant $C_{p} > 0$ to find
    \begin{equation*}
        \int_{\Omega} (u - u_{*})^{2} \leq 2 \int_{\Omega} (u - \bar{u})^{2} + 2 \alpha^{2} \int_{\Omega} \bar{m}^2 \leq 2 C_{p} \int_{\Omega} |\nabla u|^{2} + 2 \alpha^{2}m^*\int_{\Omega} m,
    \end{equation*}
    which, multiplied by $\frac{\gamma_0}{2 C_p}$ and substituted into \eqref{Lu2-a1-3}, yields
    \begin{equation}\label{Lu2-a2}
        \begin{split}
            \frac{d}{d t} \int_{\Omega} (u - u_{*})^{2} + \frac{\gamma_0}{2 C_p}\int_{\Omega} (u - u_{*})^{2} \leq & \frac{2 M^2 (\gamma_2^2+\gamma_3^2)}{\gamma_0}\int_{\Omega} (|\nabla v|^{2}+ |\nabla m|^{2})\\
            & + \frac{\gamma_0 \alpha^2 m^* + 2 \alpha M^2 C_p}{C_p}\int_\Omega m.
        \end{split}
    \end{equation}
    Moreover, we rewrite the second equation of the system \eqref{sys-1} as
    \begin{equation}\label{Lu2-a3}
        (v - \theta_{1} u_{*})_{t} = \Delta v - (v - \theta_{1} u_{*}) + \theta_{1} (u - u_{*}) + \theta_{2} u m,
    \end{equation}
    then multiply \eqref{Lu2-a3} by $(v - \theta_{1} u_{*})$ and integrate it by parts to obtain
    \begin{equation*}
        \begin{split}
            &\frac{1}{2} \frac{d}{d t} \int_{\Omega} (v - \theta_{1} u_{*})^{2} + \int_{\Omega} |\nabla v|^{2}\\
            &\leq  -\int_{\Omega} (v - \theta_{1} u_{*})^{2} +  \theta_{1} \int_{\Omega} (u - u_{*}) (v - \theta_{1} u_{*}) + \theta_{2}  \int_{\Omega} u m (v - \theta_{1} u_{*}) \\
            &\leq - \frac{1}{2}\int_{\Omega} (v - \theta_{1} u_{*})^{2}+ \frac{\theta_{1}^2}{2} \int_{\Omega} (u - u_{*})^2+ c_2\theta_2 M \int_\Omega m,
        \end{split}
    \end{equation*}
    which implies
    \begin{equation}\label{Lu2-a4}
        \frac{d}{d t} \int_{\Omega} (v - \theta_{1} u_{*})^{2} +\int_{\Omega} (v - \theta_{1} u_{*})^{2} + 2\int_{\Omega} |\nabla v|^{2}\leq \theta_{1}^2\int_{\Omega} (u - u_{*})^2+2c_2\theta_2 M \int_\Omega m.
    \end{equation}
    On the other hand, multiplying the third equation of \eqref{sys-1} by $m$ and integrating it by parts, we derive
    \begin{equation*}
        \frac{d}{d t} \int_\Omega m^2 = -2 \int_\Omega |\nabla m|^2 -2 \int_\Omega u m^2,
    \end{equation*}
    which, combined with $0\leq m(x,t)\leq m^*$, gives
    \begin{equation}\label{Lu2-a5}
        \frac{d}{d t} \int_\Omega m^2 +\int_\Omega m^2+2 \int_\Omega |\nabla m|^2\leq  m^* \int_\Omega m.
    \end{equation}
    Therefore, multiplying \eqref{Lu2-a4} and \eqref{Lu2-a5} by $\frac{M^2(\gamma_2^2+\gamma_3^2)}{\gamma_0}$ respectively, and adding them into \eqref{Lu2-a2}, we have 
    \begin{equation}\label{Lu2-a6}
        \begin{split}
            &\frac{d}{dt}\left\{\int_\Omega (u-u_*)^2+\frac{M^2(\gamma_2^2+\gamma_3^2)}{\gamma_0}\int_\Omega [(v-\theta_1 u_*)^2+ m^2]\right\}\\
            &\ \ \ +\left[\frac{\gamma_0}{2C_p}-\frac{M^2 (\gamma_2^2+\gamma_3^2)\theta_1^2}{\gamma_0}\right]\int_\Omega (u-u_*)^2+\frac{M^2(\gamma_2^2+\gamma_3^2)}{\gamma_0}\int_\Omega [(v-\theta_1 u_*)^2+m^2]\\
            &\leq \left[\frac{M^2(\gamma_2^2+\gamma_3^2)(2c_2\theta_2 M +m^*)}{\gamma_0} +\frac{\gamma_0\alpha m^*+2\alpha M^2 C_p}{C_p}\right]\int_\Omega m.
        \end{split}
    \end{equation}
    By the definitions of $M$ in Lemma \ref{LIa}, and $\gamma_0,\gamma_2$ and $\gamma_3$ in \eqref{Lu2-a1-2}, we derive that there exists a constant $\theta_*>0$ such that if $0<\theta_1< \theta_*$ that 
    \begin{equation}\label{Lu2-a7}
        M_1:=\frac{\gamma_0}{2C_p}-\frac{M^2(\gamma_2^2+\gamma_3^2)\theta_1^2}{\gamma_0}>0.
    \end{equation}
    Letting 
    $$Y(t):=\int_\Omega (u-u_*)^2+\frac{M^2(\gamma_2^2+\gamma_3^2)}{\gamma_0}\int_\Omega [(v-\theta_1 u_*)^2+ m^2],$$
    then utilizing \eqref{Lu2-a6} and noting \eqref{Lu2-a7}, we can find a constant $c_3>0$ such that
    \begin{equation*}
       Y'(t)+c_3Y(t)\leq \left[\frac{M^2(\gamma_2^2+\gamma_3^2)(2c_2\theta_2 M +m^*)}{\gamma_0} +\frac{\gamma_0\alpha m^*+2\alpha M^2 C_p}{C_p}\right] |\Omega| \|m(\cdot,t)\|_{L^\infty},
    \end{equation*}
    which, coupled with $\lim\limits_{t\to\infty}\|m(\cdot,t)\|_{L^\infty}=0$, gives
    \begin{equation}\label{Lu2-a9}
        \lim_{t\to \infty} (\|u(\cdot, t) - u_*\|_{L^2} + \|v(\cdot, t)- \theta_1 u_*\|_{L^2}) = \lim_{t\to\infty} Y(t)=0.
    \end{equation}
    Applying the parabolic regularity, we obtain that 
    \begin{equation}\label{Lu2-a10}
        \|u(\cdot,t)\|_{W^{1,\infty}}+\|v(\cdot,t)\|_{W^{1,\infty}}\leq c_4\ \ \mathrm{for} \ \ t\geq 1.
    \end{equation}
    Hence, we invoke the Gagliardo-Nirenberg inequality to conclude \eqref{Lu2-a1} by \eqref{Lu2-a9} and  \eqref{Lu2-a10}.
\end{proof}

\begin{proof}[\textbf{Proof of Theorem \ref{GBS-2}}: Stabilization.] From \eqref{Cw} and Lemma \ref{Lu2-a}, we complete the global stabilization of the constant steady state $(u_*,\theta_1 u_*,0)$.
\end{proof}
\noindent \textbf{Acknowledgment.}
The research of H.Y. Jin was supported by the NSF of China (No. 12371203), Guandong Basic and Applied Basic Research Foundation (No.
2026B0303000003).\\

\noindent \textbf{Conflict of interest}. The authors do not have any possible conflict of interest.\\

\noindent \textbf{Data availability statement}. No data sets were generated or analysed during the current study.\\

\end{document}